\documentclass[11pt]{amsart}
\usepackage[T1]{fontenc}
\usepackage[usenames]{color}
\usepackage{float,amsmath,amssymb}
\usepackage{eucal}
\usepackage{graphicx}
\usepackage[all]{xy}
\usepackage{multicol}
\usepackage{bm}
\usepackage{stmaryrd}
\newtheorem{lm}{Lemma}
\usepackage{framed}
\newtheorem{prop}{Proposition}

\newtheorem{corr}{Corollary}
\newtheorem{defi}{Definition}
\newtheorem{property}{Property}

\newtheorem{theo}{Theorem}
\newtheorem{hypo}{Assumption}

\newtheorem{rem}{Remark}

\newcommand{\Jac}{\text{Jac}}

\newcommand{\ms}[1]{\mathbb{#1}}
\newcommand{\mc}[1]{\mathcal{#1}}
\newcommand{\lip}{Lipschitz }

\newcommand{\R}{\mathbb{R}}
\newcommand{\ve}{\varepsilon}

\def\Span{\text{Span}}
\def\1{\mathchoice {\rm 1\mskip-4mu l} {\rm 1\mskip-4mu l}
{\rm 1\mskip-4.5mu l} {\rm 1\mskip-5mu l}}
\begin{document}
\author{Fran\c{c}ois-Xavier Vialard}
\address[Fran\c{c}ois-Xavier Vialard]{Institute for Mathematical Science\\ Imperial College London\\
53 Prince's Gate, SW7 2PG, London, UK}
\email{f.vialard@imperial.ac.uk}
\title[Stochastic Shape Splines]{Extension to Infinite Dimensions of a Stochastic Second-Order Model associated with the Shape Splines}
\date{\today}
\begin{abstract}
We introduce a second-order stochastic model to explore the variability in growth of biological shapes with applications to medical imaging. Our model is a perturbation with a random force of the Hamiltonian formulation of the geodesics. Starting with the finite-dimensional case of landmarks, we prove that the random solutions do not blow up in finite time. We then prove the consistency of the model by demonstrating a strong convergence result from the finite-dimensional approximations to the infinite-dimensional setting of shapes. To this end we introduce a suitable Hilbert space close to a Besov space that leads to our result being valid in any dimension of the ambient space and for a wide range of shapes. 
\end{abstract}

\maketitle
\tableofcontents

\section{Introduction}
A new problem has emerged very recently in computational anatomy: the mathematical modeling and the statistical study of biological shape changes. Medical applications are of great interest such as the early detection of disease: for instance, Alzheimer's disease induces hyppocampal atrophy. Current approaches study the shape evolution through indicators (such as the volume or the length of characteristic patterns) or through the parameters of objects with simple geometries (such as ellipsoids) used to describe the more complex biological shape of interest.
Therefore there is still room for a more quantitative analysis of the variability of longitudinal data.

Although the analysis of shape evolution is quite a new question, the analysis of the variability of static shapes has motivated tremendous research in recent years with broad applications in medical imaging. Most of the efforts has been on developing tools to compare two static shapes. This problem is also refered to as the registration problem. Numerous attempts to answer this problem introduce a metric on the space of shape \cite{michor-2007-23,michor-2007} and the geodesic flow \cite{1122513} on this space provides a powerful framework to statistically study the variability of biological organs among a population \cite{vmty04}. Hence it seems reasonable to build out of this framework proper tools to analyse the growth of shapes.
\\
Our work contributes to the field of \emph{large deformation by diffeomorphisms} that emerged twenty years ago with the idea of studying shapes under the action of a group of transformations of the ambient space \cite{gren1}. Thus the distance on the space of shapes is induced by the distance on a group of diffeomorphisms through its action on shapes \cite{0855.57035}. This framework has been widely applied to computational anatomy in the recent years \cite{Younes2009S40,BegIJCV} and important contributions have been made even on numerical issues \cite{1751-8121-41-34-344003}. Different ways of representing shapes have been introduced to fit in this framework, such as points of interest (also refered to as landmarks), measures or currents for surfaces \cite{JoanPhD}. The application of the theory to images is also important since it can avoid pre-segmentation operations that erase information. The approach of \emph{large deformation by diffeomorphisms} has therefore proven to be adaptable and powerful. 
\\
In attempting to describe the growth of biological organs, non-diffeomorphic evolutions should be taken into account at some point. The so-called metamorphoses framework \cite{trouve05:_local_geomet_defor_templ,DBLP:journals/corr/abs-0806-0870} can deal with such evolutions. However we will focus on the diffeomorphic case which is the first step to be understood.

The initial registration problem on images aims at minimizing a functional (see formula \eqref{matching_chap5}) which is the sum of two terms: the first being the cost of the transformation and the second being a similarity measure between the transformed shape $\phi_1.S_0$ and the target $S_{target}$.
\begin{equation} \label{matching_chap5}
\mathcal{J}(u) = \int_0^1 \| u_t \|_V^2 dt + d(\phi_1.S_0,S_{target}) \, ,
\end{equation}
In this equation, $u_t \in L^2([0,1],V)$ is a time-dependent vector field where $V$ is a reproducing kernel Hilbert space of vector fields and $d$ is a distance on the space of shapes. The group of diffeomorphisms is generated by the flow at time $1$ of such time dependent vector fields and an action of this group on the space of shapes. We have denoted the action of $\phi_1$ on $S_0$ by $\phi_1.S_0$. If there then exists a minimum to this functional, this minimum will provide a balance between a good matching of the target shape and the cost of the transformation.
\\
The first term on the right-hand side in \eqref{matching_chap5} should reflect the likelihood of the deformation $\phi_1$. Therefore this matching procedure is motivated by a Bayesian approach as presented in \cite{Dupuis1}: the starting point is to interpret the minimization of the functional \eqref{matching_chap5} as a maximum a posteriori. Although in \cite{Dupuis1} no rigorous results were established to make the connection between the MAP interpretation and the minimisation problem, a rigorous asymptotic theory of the problem was developed recently in \cite{largedev}. The author prove a large deviation principle for which the rate is the first term of \eqref{matching_chap5}. The probabilistic foundations for their work are given by the study of stochastic flows of diffeomorphisms developed in \cite{Kunita} and the random object associated with the prior on the diffeomorphism in \eqref{matching_chap5} is the stochastic flow defined by,
\begin{equation}
\begin{aligned}
& \phi_t(x) = \int_0^t W_s(\circ ds , \phi_s(x)) \, , \\
& W_t = \sum_{i=0}^{\infty} B_i(t)e_i \, .
\end{aligned}
\end{equation}
In these equations $(B_i)_{i \in \ms{N}}$ are i.i.d Brownian motions, $(e_i)_{i \in \ms{N}}$ is an orthonormal basis of $V$ and the symbol $\circ$ stands for the Stratonovich integral. If the space of vector fields is smooth enough (regularity assumptions on the kernel), the proof of the existence of the stochastic flow can be found in Theorem 4.6.5 of \cite{Kunita}. Through the action of the random diffeomorphisms, this approach gives evolutions of the shape that are non- smooth in time due to the Brownian motions. However, at each time the transformation is smooth in space as we can see in figures Fig.~\ref{Kunita1} and Fig.~\ref{Kunita2}. In these two figures, we present the time evolution ($z$-axis) of $40$ points on the unit circle under the transformation of a Kunita flow of diffeomorphisms for a Gaussian kernel of width $0.9$.

\begin{figure}[htbp]
 \begin{minipage}{.45\linewidth}
  \centering	\includegraphics[width=7cm]{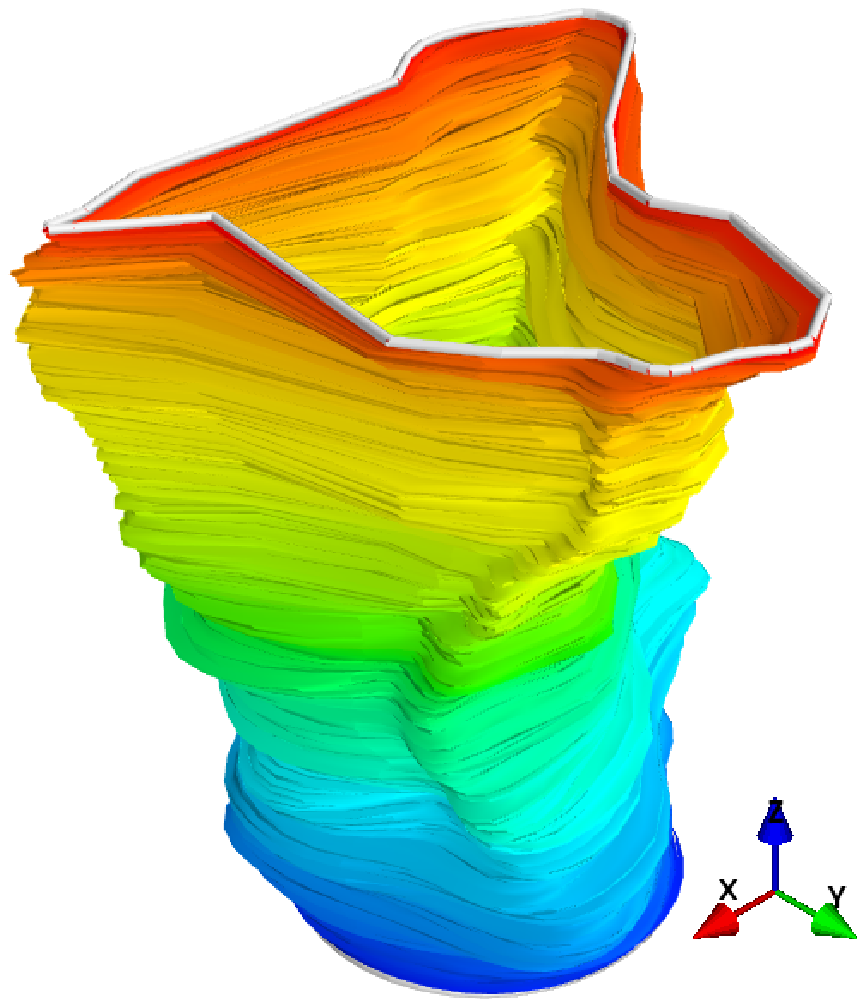}
 \caption{\footnotesize{Simulation of Kunita flow - $40$ points on the white unit circle as initial shape.}}
 \label{Kunita1}
\end{minipage} \hfill
\begin{minipage}{.5\linewidth}
\centering\includegraphics[width=7cm]{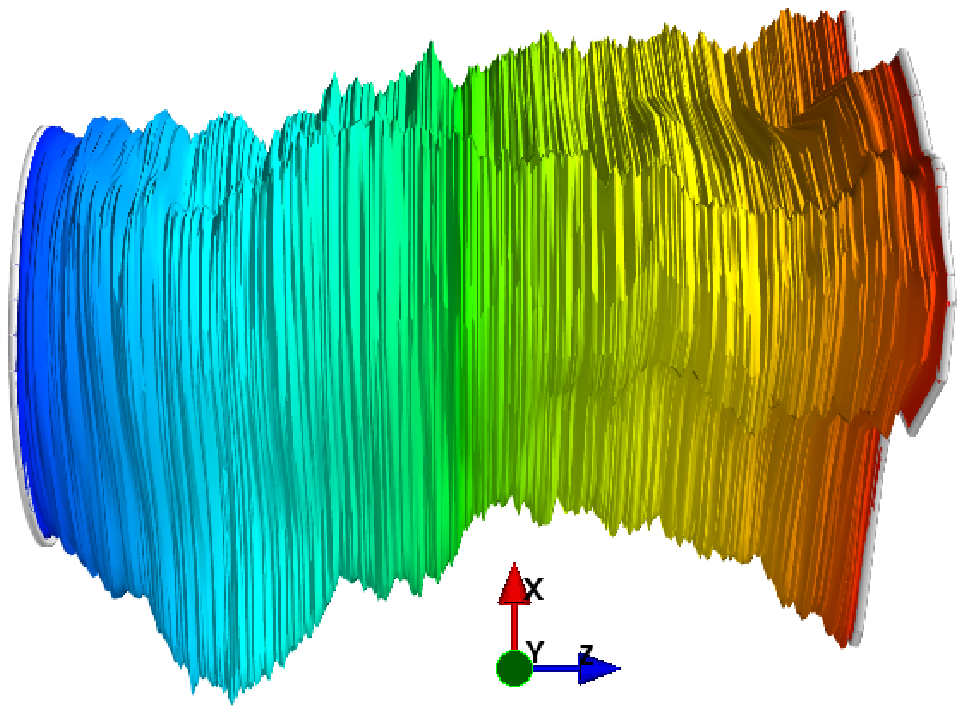}
\caption{\footnotesize{Simulation of Kunita flow - the time axis is $z$, the blue arrow.}}
  \label{Kunita2}
 \end{minipage} \hfill
\end{figure}

We would definitely prefer a probabilistic framework for smooth evolutions of shapes that we think are closer to biological growth evolutions (note that this hypothesis is heuristic). In addition, the large deviation result in \cite{largedev} does not lead to a generative model for diffeomorphic evolutions in this context. One important property of the model would be the smoothness (i.e. not as rough as a standard path of the Brownian motion) of trajectories. 
\\
\emph{To build a second-order model coherent with the framework developed for diffeomorphic matching is a natural way to overcome this issue.}

Let us discuss the finite-dimensional case of particles. In the landmark case, the minimization of \eqref{matching_chap5} reduces to the calculation of the geodesic flow on a Riemannian manifold. The Hamiltonian formulation of this geodesic flow is often used in practical applications \cite{bb37524} and seems appealing as well to build this second-order model. To describe time dependent evolutions we can introduce a control term in the equation of the momentum that would guide the trajectory to match the evolution. Minimizing an energy term on the control variable would lead to a generalized version of the splines on a Riemannian manifold pioneered in \cite{Noakes1}. We have recently studied this model that we called splines on shape spaces in \cite{ShapeSplines} focusing on the finite-dimensional case.
The new evolution equations on the Riemannian manifold $M$ of landmarks (i.e. $q\in M$ stands for a group of points) are then
\begin{equation} \label{control}
\begin{cases}
\dot q = -\partial_q H \\
\dot p = \partial_p H + u\,,
\end{cases}
\end{equation}
and we aim at minimizing 
\begin{equation} \label{spline}
\mathcal{E}(u) = \int_0^T g(u_t,u_t) \, dt + \sum_{i=1}^n|q_{t_i} - x_{t_i}|^2 \,,
\end{equation}
where $g$ is a metric that measures the cost of the forcing term $u_t \in T^*M$. The random object associated with this model is obtained by replacing $u_t$ with a standard white noise. Hence it can give a reasonable stochastic model to generate $C^1$ trajectories in $M$. This stochastic model seems promising to study since we expect to keep the numerical tractability allowed by the Hamiltonian formulation. 
\\
However the main interesting feature concerning the modeling aspect of our work is the physical interpretation of these equations. If we consider the evolution of landmarks as a physical system of particles, it seems natural to introduce a random force to their evolution: an additive white noise is added to the evolution equation of the momentum. This idea of perturbing the evolution equations with a random force has been introduced for a long time in the stochastic fluid dynamics community (\cite{Bensoussan}). Our stochastic system is a stochastic perturbation of Euler-Poincaré equation coding for the geodesics on a group of diffeomorphisms (also refered to as EPDiff equation, \cite{citeulike:620104}). From this point of view, this work may have some relations with the study of stochastic perturbations of the vortex model (\cite{anna} and \cite{fluidspde} for a brief survey). We do not develop these links in this work but instead we will focus on this model of growth of shape. However, to turn this model into a tractable candidate to deal with a collection of shape evolutions at different times and to perform statistical studies on real data, we would need to introduce a drift term (i.e. a deterministic forcing term) in the momentum equation. Finally, if the stochastic model is well-posed when there is no forcing term, it will not be difficult to extend it.
\\
By well-posed, we mean it possesses the following two features:
\begin{itemize}
\item the existence for all time of the solutions of the stochastic equations (since the Hamiltonian system does not have linear growth),
\item the extension of the model to infinite dimensions (on shape spaces) and associated convergence results. 
\end{itemize}
In this work, we answer both questions in the affirmative and the strategies followed are the classical ones:
for the non-blow-up result, the application of the It\^{o} formula gives a linear control on the expectation of the energy of the system measured by the Hamiltonian. The extension to infinite dimensions relies on the construction of a new Hilbert space that is close to Besov space. This Hilbert space is much more tractable than the classical Sobolev or Besov spaces and it suits perfectly our convergence result that is somehow disconnected from the chosen Hilbert space for the approximation.

\vspace{0.2cm}
\noindent
The paper is organized as follows: after an introduction to the deterministic case in Section \ref{overview} in which we present the convergence to the infinite-dimensional case of curves, we prove in Section \ref{landmarkcaseSection} that the SDE in the finite-dimensional case of landmarks has solutions for all time. In Subsection \ref{prop_needed}, we prove the property that the shape space should fulfill in order that the SDE in infinite dimensions is well defined. To give an example of such a space, we introduce in Section \ref{F_s} a new Hilbert space $F_s$ which has interesting properties of stability under composition with smooth functions, product stability and which also contains smooth functions. 
\\
We define the cylindrical Brownian motion in Section \ref{noise_intro} and the It\^{o} integral in a useful way for the convergence results developed in Section \ref{generalsolutions}. These results rely on approximation lemmas detailed in Section \ref{Approx_lemmas} that are somehow disconnected from the finite-dimensional approximations. Section \ref{applications} draws on the previous sections to illustrate applications of this convergence results. We also show some numerical simulations. Finally Section \ref{conclusion} tries to open research directions around this stochastic second-order model essentially motivated by applications.

\section{Overview of the deterministic case} \label{overview}
\subsection{Optimal control heuristic}
In this section we present an optimal control heuristic to derive the Hamiltonian equations that can be formally applied in the finite or infinite-dimensional case. These results are proven in \cite{laurentbook} for the case of landmarks and in \cite{hamcurves} in the case of curves.

Let $G$ be the group of diffeomorphisms of $\ms{R}^d$ ($d$ is the dimension of the ambient space) generated by the flows of time dependent vector fields in $L^2([0,1],V)$ where $V$ is a Reproducing Kernel Hilbert Space (RKHS) of $C^1$ vector fields on $\ms{R}^d$ with the additional hypothesis:

\begin{hypo} \label{inj1}
There exists a continuous injection of the Hilbert space $V$ of vector fields into $C^1$, i.e. there exists a positive constant $K$ such that
$|v|_{1,\infty} \leq K |v|_V$ for any $v \in V$.
\end{hypo}
We also say that $V$ is $1-$admissible. It is more demanding than the RKHS condition, namely that the pointwise evaluation is a continuous form on $V$. In what follows, $k:\ms{R}^d\times \ms{R}^d \to L(T^*\ms{R}^d) $ will denote the kernel of the RKHS.
\\
Then the minimization problem \eqref{matching_chap5} can be recast into an optimal control problem:
if the space of shapes is a Banach space $E$ endowed with an action of the group $G$ that we assume to be differentiable in the following sense:
There exists a linear map 
\begin{gather*}
V \times E  \mapsto E \\
(v , q) \mapsto v . q
\end{gather*}
such that for any $v \in L^2([0,1],V)$ we have
\begin{align*} 
& \frac{d}{dt} [\phi_{0,t}.q] = v_t . [\phi_{0,t} . q] \, a.e., \\
& \phi_{0,1}.q = q + \int_0^1 v_t . [\phi_{0,t} . q] \, dt \, \forall t \in [0,1].
\end{align*}
The existence of a minimizer for the functional \eqref{matching_chap5} in situations of interest is usually proven via standard arguments of lower semi-continuity for the weak topology on $L^2([0,1],V)$. Let us assume that there exists a minimizer $v_0$ for the functional \eqref{matching_chap5}, then it also minimizes the energy $\frac{1}{2}\int_0^1 |v_t|^2_{V} \, dt$ with fixed endpoint $\phi_{0,1}^v.q_0 = \phi_{0,1}^{v_0}.q_0$. The optimal control theory enables us to be a little more general by assuming that we are interested in the solutions of the minimization of:
\begin{equation} \label{optimalproblem}
\begin{cases}
& \inf \frac{1}{2}\int_0^1 |v_t|^2 dt \\ 
& q(0) \in M_0 \\
& q(1) \in M_1 
\end{cases}
\end{equation}
with $M_0,M_1$ two subsets of $Q$ with tangent spaces at a point $q_i \in M_i$ denoted by $T_{q_i}M_i$ for $i = 0,1$. The case of $M_0$ and $M_1$ can be found in the case of curves considered up to reparameterization as demonstrated in \cite{reparam}: the two subsets $M_i$ for $i \in \{0,1\}$ are generated by the action of the group of diffeomorphisms of $S_1$ and the functional \eqref{matching_chap5} may be invariant for this action. In that case, the Pontryagin Maximum Principle (PMP, \cite{MR2062547,trelatbook}) provides orthogonality relations for the momentum.
With $E^*$ the dual of $E$, the control on $q \in E$ is $v \in V$ with an instantaneous cost function $\frac{1}{2}|v|_V^2$ and we have $\dot q = v.q$. Then, the Hamiltonian system associated with this minimization problem is
\begin{equation} \label{Hamiltoncontrol}
H(p,q,v) = (p,v.q)_{E^*,\,E} - \frac{1}{2} \langle v, v \rangle_V \, .
\end{equation}
Before minimizing in $v$, we need to assume that 
\begin{gather*}
V \mapsto \ms{R} \\
v \mapsto (p,v.q)_{E^*,\,E}
\end{gather*}
is a continuous linear form on $V$ for any $p,q \in E^* \times E$ (hypothesis (H1)). 
\\
For example in the case of landmarks the hypothesis (H1) just says that the pointwise evaluation on $V$ is continuous, i.e. $V$ is a RKHS of vector fields. Then, by the Riesz theorem there exists $ p \diamond q \in V^*$ defined by the equation:
$$(p,v.q)_{E^*,\,E} = -(p \diamond q,v)_{V^*,V} \, .$$ 
This notation is taken from \cite{DBLP:journals/corr/abs-0806-0870} and it is also known as the momentum map in geometric mechanics.

The second term of equation \eqref{Hamiltoncontrol} can be rewritten as $(Lv,v)_{V^* \times V}$.
Then, at a minimum we can differentiate in $v$ to obtain 
$$
Lv + p \diamond q = 0 \, 
$$
or equivalently, $v  + K(p \diamond q) = 0 $.

Therefore the 'minimized' Hamiltonian is,
\begin{equation} \label{minHam}
H(p,q) = (p \diamond q , K p \diamond q)_{V^*,V} - \frac{1}{2} \langle v, v \rangle_V  = \frac{1}{2} (p \diamond q , K p \diamond q)_{V^*,V} \, .
\end{equation}
The Pontryagin maximum principle says that a minimizer of the problem \eqref{optimalproblem} verifies the following Hamiltonian system 
\begin{equation} \label{Controlhamiltonsystem}
\begin{cases} 
\dot p = - \partial_q H(p,q) \\
\dot q = \partial_p H(p,q) \, , \nonumber
\end{cases}
\end{equation}
with orthogonality conditions
\begin{gather*}
p(0) \perp T_{q(0)} \, ,\\
p(1) \perp T_{q(1)} \,.
\end{gather*}
At this point, we need to give a sense to $\partial_q H(p,q)$ in \eqref{Controlhamiltonsystem}. Assuming that $q \mapsto (p,v.q)$ is differentiable for every $q \in E$, we write
$\delta q \mapsto \partial_q (p,v.q)(\delta q)$. In addition, we assume that $\partial_q (p,v.q)(\delta q)$ is a linear form on $V$ (hypothesis (H2)) that we denote 
$$\partial_q (p,v.q)(\delta q) = -(\partial_q (p \diamond q)(\delta q),v)_{V^* \times V} \, .$$ 
The differentiation of $H(p,q)$ reads,
$$ \partial_q H(p,q) = (\partial_q (p \diamond q), K (p \diamond q)) \, .$$
In the landmark case, the second hypothesis (H2) says that the pointwise evaluation for the first derivative of the vector fields is continuous on $V$. In particular, if $V$ is $1-$admissible, this condition is fulfilled.

We now present the case of landmarks that will be the cornerstone of this work. The Hamiltonian system reads, if $q \doteq (q_i)_{i \in [1,n]}$ are the particles in $\ms{R}^d$ and $p \doteq (p_i)_{i\in [1,n]}$ are the associated momentums
\begin{equation} \label{LandmarkCaseDeteministe}
\begin{cases}
\dot{p_i} = - \sum_{j=1}^n \partial_1 k(q_i,q_j) \langle p_i,p_j \rangle_{\ms{R}^d} \,,\\
\dot{q_i} = \sum_{j=1}^n k(q_i,q_j)p_j \,.
\end{cases}
\end{equation}
This is the Hamiltonian system that will be perturbed in Section \ref{landmarkcaseSection}.

Let us discuss the case of curves following the point of view adopted in \cite{hamcurves}. We consider generalized closed curves, which means that  we will work on $Q = L^2(S_1,\ms{R}^2)$. Of course, this framework can be extended to $Q = L^2(M,\ms{R}^d,\mu)$ with $M$ a compact Riemannian manifold and $\mu$ its associated measure, for instance $M=S_n$ the $n$-dimensional sphere or $M=T_n$ the $n$-dimensional torus. The action by $G_V$ is simply the left composition on $Q$.
The map on $V\times Q$ induced by the action of $G_V$ is also the left composition with $v \in V$: $(v,q) \mapsto v \circ q$ and the hypothesis (H1) is verified if $V$ is an admissible space of vector fields since:
$$ \int_{M} \langle v(q(s)),p(s) \rangle d\mu(s) \leq |p|_{L^2} |v \circ q|_{L^2} \leq |p|_{L^2} |v|_{\infty} \sqrt{\mu(M)} \, .$$ 
The second hypothesis (H2) is verified in this case too replacing $v$ by $dv$. Remark that $[\partial_q v \circ q].\delta q = [dv \circ q] (\delta q) \in L^2(M,\ms{R}^d,\mu)$.

\vspace{0.2cm}
\noindent
The transversality conditions are interesting in the case of curves considered up to reparameterization. If the initial curve $c_0$ and the final one $c_1$ are smooth enough, the action of the diffeomorphism group of $S_1$ generates a large subspace of tangent vectors at $c_i$ for $i=0,1$: let $w$ be a smooth vector field on $S_1$, if $\psi_t$ is the flow generated by $w$, we have $\frac{d}{dt}_{t=0} \phi_t(x) = w(x)$ then,
$$ \{ s \mapsto w(s) c_i'(s) \,  |  \, w \in \mc{X}^{\infty}(S_1) \} \subset T_{c_i}\, .$$
The orthogonality condition says that $p_i \perp T_{c_i}$ and considering all the choices for $w$ (any smooth vector field on $S_1$) we obtain that 
$$ \langle p_i(s), c_i'(s) \rangle = 0 \; a.e. \, s \in S_1 \, .$$

\subsection{Convergence to the infinite-dimensional case}

We develop a consequence of the Hamiltonian formulation of the equations originally written in \cite{hamcurves}, not written in this article. This paper presents a rigorous proof of the existence in all time of the solutions to the Hamiltonian equations when the space of closed curves is the Hilbert space $H=L^2(S_1,\ms{R}^2)$ and the momentum variable lies in the dual space of $H$, identified to $H$. The structure of the momentum variable is determined by the differentiation of the attachment term in \eqref{matching_chap5} and the situation $p \in H$ arises for a large class of attachment term.
The Hamiltonian system 
\begin{subequations} \label{syst}
\begin{align}
\partial_q H(p_t,q_t) &= -p_t(.) \int_{S_1} \partial_1 k(q_t(.),q_t(s))p_t(s)ds\,, \label{eq1} \\
\partial_p H(p_t,q_t) &= \int_{S_1} k(q_t(.),q_t(s))p_t(s)ds\,, \label{eq2}
\end{align}
\end{subequations}
has solutions for all time for any initial conditions $(p_0,q_0) \in H^2$. 

\noindent
A simple though important remark is 
\begin{rem}
The ODE \eqref{syst} conserves the common structure of $p$ and $q$: i.e. if $p$ and $q$ are both constant (in space) on an interval (resp. a measurable set on $S_1$) then the solution $(p,q)$ will be constant on this interval (resp. on this measurable set).
\end{rem}
The consequence of this remark is that the landmark case is a special case of the ODE \eqref{syst}. Consider the $n$-dimensional subspace of $H$ for $n \geq 1$, $H_n=\Span_{i \in [0,n-1]} (1_{[\frac{i}{n},\frac{i+1}{n}[})$, then $H_n$ is one candidate to describe the trajectories of $n$ landmarks, taking initial conditions $(p_0,q_0) \in H_n \times H_n$. It gives also a convergence property by the continuity of solutions of a \lip ODE system. With stronger assumptions on the convergence of $q^n$ but still the same assumption on the convergence for $p^n$, we obtain strong convergence of $q^n$.

\begin{prop}
Let $(p_0^n,q_0^n) \in H_n \times H_n$ be initial conditions for the system \eqref{syst} with $\lim_{n \mapsto \infty} (p_0^n,q_0^n) = (p,q)$ then, the solutions $(p^n_t,q^n_t)$ converge in $H \times H$ to $(p_t,q_t)$ uniformly for $t$ in a compact set.
If we assume in addition that $\lim_{n \mapsto \infty} \| q^n -q \|_{\infty} = 0$ then $\lim_{n \mapsto \infty} \| q^n_t -q_t \|_{\infty} = 0$ uniformly for $t$ in a compact set.
\end{prop}

\begin{proof}
The first point is the direct application of the continuity theorem for the Banach fixed point theorem with parameter. The second point is a consequence of the first one: since 
\begin{equation}
|v\circ q_n - v\circ q|_{L^2} \leq |dv|_{L^{\infty}}|q_n-q|_{L^2} \leq K |v|_V |q_n-q|_{L^2} \,,
\end{equation}
it implies that $p_n \diamond q_n$ is bounded in $V$ and it is weakly convergent to $p \diamond q$. Then the convergence in $V$ and Assumption \ref{inj1} implies the convergence in $L^{\infty}$ uniformly for $t$ in a compact set.
\end{proof}

\section{The stochastic model for landmarks} \label{landmarkcaseSection} 
The simplest perturbation of the deterministic Hamiltonian equations to obtain a second-order stochastic model is the addition of a white noise in the momentum equation. Closely related to splines on shape spaces introduced in \cite{ShapeSplines}, this stochastic model is also presented but only in the finite-dimensional case.

\begin{subequations} \label{M1}
\begin{gather}
dp_t = -\partial_q H(p_t,q_t) \, dt + \ve dB_t \, ,   \\ 
dq_t = \partial_p H(p_t,q_t) \, dt\,. 
\end{gather}
\end{subequations}
\noindent
Here, $\ve$ is a positive real parameter and $B_t$ is a Brownian motion on $\R^{dn}$ and we can think of the kernel as a diagonal kernel, for instance the Gaussian kernel or the Cauchy kernel (which verify hypothesis (H1) and (H2)). To study this SDE, we will use the It\^{o} stochastic integral. From the theorem of existence and uniqueness of solutions of stochastic differential equation under the condition of linear growth, we can work on the solutions of such equations for a large range of kernels. However in our case the Hamiltonian is quadratic, and the classical results for existence and uniqueness of stochastic differential equations only prove that the solution is locally defined. In the deterministic case, the Hamiltonian which represents the energy of the system remains constant along the geodesic paths. By controlling the Hamiltonian of the stochastic system, we will prove that the solutions are defined for all time. 
\\ 
First remark that if $x \in \ms{R}^d$ and $\alpha \in \ms{R}^d$ then 
\begin{equation}\label{IneqRKHS}
\langle \alpha , k(x,x)\alpha \rangle_{\ms{R}^d} \leq K^2 |\alpha|_{\ms{R}^d}^2 \,.
\end{equation}
Now we introduce the stopping times defined as follows: let $M>0$ be a constant and
\begin{equation} \label{stoppingtimes}
\tau_M = \{ t \geq 0 \, | \,\max(|q_t|,|p_t|) \geq M \} \, , 
\end{equation}
let also $\tau_\infty=\lim_{M\to\infty}\uparrow
\tau_M$ be the explosion time. 
\\
Differentiating $H(p_{t \wedge \tau_M},q_{t \wedge \tau_M})$ with respect to $t$, we get on $(t<\tau_M)$:
\begin{equation*}
dH(t) = \partial_q H(p_t,q_t)dq_t + \partial_p H(p_t,q_t)dp_t  + \frac{\ve^2}{2} \sum_{i=1}^n \text{tr}(k(q_i(t),q_i(t))) dt \,.
\end{equation*}
In the deterministic case the Hamiltonian is constant, whereas the stochastic perturbation gives
$$ \partial_q H(p_t,q_t)dq_t + \partial_p H(p_t,q_t)dp_t = \ve \partial_p H(p_t,q_t) dB_t \,.$$
\begin{align}
&\int_0^{T \wedge \tau_M} dH(t) = \int_0^{T \wedge \tau_M} \ve \langle \partial_p H(p_t,q_t), dB_t \rangle  + \int_0^{T \wedge \tau_M} \frac{\ve^2 }{2}\sum_{i=1}^n \text{tr}(k(q_i(t),q_i(t))) dt  \, ,  \nonumber\\
& E[H(p_{T \wedge \tau_M},q_{T \wedge \tau_M})] \leq H(0) + E(\frac{\ve^2}{2} dn \, T \wedge \tau_M) \leq H(0) + (K\ve )^2 dn T \, .  \label{eq:1} 
\nonumber
\end{align}
Now, we aim at controlling $q_{t \wedge \tau_M}$ using the control on $dq_t$ given by $|\partial_p H(p_t,q_t)|_{\infty} \leq K \sqrt{H(p_t,q_t)}$:
  \begin{multline}
    |q_{\tau_M\wedge t}|\leq |q_0|+\int_0^{\tau_M\wedge
      t}KH(p_s,q_s)^{1/2}ds
\leq |q_0|+\int_0^{\tau_M\wedge
      t}KH(p_{s\wedge \tau_M},q_{s\wedge\tau_M})^{1/2}ds\\
\leq  A_t\doteq  |q_0|+\int_0^{\tau_\infty\wedge
      t}KH(p_{s\wedge \tau_\infty},q_{s\wedge\tau_\infty})^{1/2}ds\,.
  \end{multline}
However, $0\leq A_t$ $P$ a.s.  and 
by monotone convergence theorem (recall that $H$ is non-negative),
$$E(A_t)=\lim_{M\to\infty}(|q_0|+E\left(\int_0^{t\wedge\tau_M} KH(p_{s\wedge
  \tau_M},q_{s\wedge \tau_M})^{1/2}ds \right).$$
Also,
 \begin{multline}
E\left(\int_0^{t\wedge\tau_M} H(p_{s\wedge
  \tau_M},q_{s\wedge \tau_M})^{1/2}ds\right)\leq E\left(\int_0^{t} H(p_{s\wedge
  \tau_M},q_{s\wedge \tau_M})^{1/2}ds\right)\\
\stackrel{Fub.}{=}\int_0^{t} E\left(H(p_{s\wedge
  \tau_M},q_{s\wedge \tau_M})^{1/2}\right)ds\stackrel{Jen.}{\leq}\int_0^{t} E\left(H(p_{s\wedge
  \tau_M},q_{s\wedge \tau_M})\right)^{1/2}ds\\
\stackrel{CS+\eqref{eq:1}}{\leq}\sqrt{t}\left(\int_0^t(H(0)+(K\varepsilon)^2 nds)\,
  ds\right)^{1/2} \, . \label{eq:1'}
\end{multline}  
We deduce 
$$E(A_t)\leq |q_0|+K\sqrt{t}\left(\int_0^t(H(0)+(K\varepsilon)^2 nds)\,
  ds\right)^{1/2}<\infty\text{ and }A_t<\infty\ P\ a.s.$$
and as a consequence $$\limsup_{M\to\infty} |q_{t\wedge\tau_M}| < +\infty \, P \, a.s.$$
We also control the evolution equation of the momentum as follows,
\begin{equation} \label{unederniere}
 |p_{t \wedge \tau_M}| \leq \int_0^{t \wedge \tau_M} |\partial_q H(p_s,q_s)| \, ds + | p_0 + \,\int_0^{t \wedge \tau_M} \ve dB_s \, | \, .
\end{equation}
Now we use the assumption \ref{inj1} to control $\partial_q H(p,q)$:
\begin{equation*}
|\partial_q H(p,q)| \leq |p| |dv(q)| \leq K|p| H^{1/2} \, .
\end{equation*}
We rewrite inequality \eqref{unederniere} and we use Gronwall's Lemma to get:
\begin{align*}
& |p_{t \wedge \tau_M}| \leq \int_0^{t \wedge \tau_M}  K|p_s| \, H(p_{s},q_{s})^{1/2} \, ds + | p_0 + \,\int_0^{t \wedge \tau_M} \ve dB_s \, | \, , \\
& |p_{t \wedge \tau_M}| \leq \left(|p_0| +  \sup_{u \leq t} | \,\int_0^{u \wedge \tau_M} \ve dB_s \, | \right) e^{\int_0^{t \wedge \tau_M} K H(p_s,q_s)^{1/2} ds } \, , \\
& |p_{t \wedge \tau_M}| \leq \left(|p_0| +  \sup_{u \leq t\wedge \tau_{\infty}} | \,\int_0^{u} \ve dB_s \, | \right) e^{\int_0^{t \wedge \tau_{\infty}} K H(p_s,q_s)^{1/2} ds } \, .
\end{align*}
The first term on the right-hand side $|p_0| +  \sup_{u \leq t\wedge \tau_{\infty}} | \int_0^{u} \ve dB_s \, |$ is bounded by $ |p_0| +  \sup_{u \leq t} | \,\int_0^{u} \ve dB_s \, | < \infty \, P \, a.s.$ and with inequality \eqref{eq:1'} we have that $$  e^{\int_0^{t \wedge \tau_{\infty}} K H(p_s,q_s)^{1/2} ds} < \infty \, P \, a.s. $$
Since on $(\tau_\infty\leq t)$ one
has $$\lim_{M\to\infty}\max(|q_{t\wedge\tau_M}|,|p_{t\wedge
  \tau_M}|)= \lim_{M \to \infty} |p_t| = \infty \, ,$$ we deduce $P(\tau_\infty\leq t)=0$ and
$\tau_\infty=+\infty$ almost surely.

\vspace{0.3cm}

\noindent
We have proven for the case $\ve(p,q) = \ve Id$,

\begin{prop} \label{landmarkcase}
Under assumption \ref{inj1} and if $\ve: \ms{R}^{nd} \times \ms{R}^{nd} \mapsto L(\ms{R}^{nd})$ is a \lip and bounded function, the solutions of the stochastic differential equation defined by
\begin{eqnarray*} \label{sytem_sto_land}
 dp_t &=& -\partial_q H(p_t,q_t)dt + \ve(p_t,q_t) dB_t   \\ 
 dq_t &=& \partial_p H(p_t,q_t)dt. 
\end{eqnarray*}
do not blow up in finite time a.s.
\end{prop}

\begin{proof}
To extend the proof to the case when $\ve$ is a \lip and bounded function of $p$ and $q$, we just prove that the preceding inequalities are still valid.
\\
First, thanks to the \lip property of $\ve$ the solutions are still defined locally. The It\^{o} formula now reads, on $(t < \tau_M)$
\begin{equation*}
dH(t) = \partial_x H(p_t,x_t)dx_t + \partial_p H(p_t,x_t)dp_t + \frac{1}{2}\text{tr}(\ve^T(p_t,x_t)K_{x_t}\ve(p_t,x_t))dt \,.
\end{equation*}
where $k_{x}$ is block matrix defined by $k_{x}\doteq [k(x_i,x_j)]_{1\leq i,j\leq n}$.
\noindent

We still have the inequality \eqref{eq:1} with 
$$\text{tr}(\ve^T(p_t,x_t)k_{x_t} \ve(p_t,x_t)) \leq (Knd|\ve|_{\infty})^2$$ 
if $|\ve(p,x) w|^2 \leq |\ve|_{\infty}^2 |w|_{\infty}^2$ where $|\ve|_\infty$ denotes the supremum norm. Indeed, if $(e_i)_{i\in [1,nd]}$ the canonical basis of $\ms{R}^{nd}$, denoting $\ve\doteq\ve(x,p)$, we have 
\begin{equation*}
\text{tr}(\ve^t k_x \ve)  = \sum_{i=1}^{nd} \langle \ve(e_i),k_x\ve(e_i) \rangle  
\leq \lambda^{*}_{k_x} \sum_{i=1}^{nd} \langle \ve(e_i), \ve(e_i) \rangle\,, 
\end{equation*}
where $\lambda^{*}_{k_x}$ is the largest eigenvalue of $k_x$. We have $\lambda^{*}_{k_x}\leq \text{tr}(k_x)=\sum_{i=1}^n 
\text{tr}(k(x_i,x_i))$ and using (\ref{IneqRKHS}) we get $\text{tr}(k(x_i,x_i))\leq d\lambda^*_{k(x_i,x_i)}\leq dK^2$ so that $\lambda^{*}_{k_x}\leq K^2nd$. Hence,
$$
\text{tr}(\ve^t k_x \ve)  \leq K^2 nd \sum_{i=1}^{nd} |\ve|_{\infty}^2 \leq (Knd|\ve|_{\infty})^2 \,.
$$ 
Thus we get,
\begin{align*}
&\int_0^{T \wedge \tau_M} dH(t) \leq \int_0^{T \wedge \tau_M} \langle \partial_p H(p_t,x_t), \ve(p_t,x_t)dB_t \rangle + \int_0^{T \wedge \tau_M} \frac{(Knd|\ve|_{\infty})^2}{2}dt \, , \\
& E[H(T \wedge \tau_M)] \leq H(0) + E(\frac{(Knd|\ve|_{\infty})^2}{2} \, T \wedge \tau_M) \leq H(0) + (Knd|\ve|_{\infty})^2\,T \,,
\end{align*}
and all the remaining inequalities follow easily with the control on $H$ and the bound on $\ve$.
\end{proof}

The first comment is that this model perturbs as expected the trajectories of the deterministic system and the model provides realistic evolutions for biological shapes contrary to Kunita flows. When $\ve \rightarrow 0$, the solutions of the system \eqref{sytem_sto_land} converge to the corresponding geodesic $\ve = 0$. The first simulation\footnote{We used a simple Euler scheme to simulate the SDE} in figure Fig.~\ref{GeodesicEvolution} represents the geodesic evolution of a template ($40$ equidistributed points on the white unit circle) to the target ($40$ points on the white ellipse deduced of the initial points with affine transformation), the color change from blue to red represents the time evolution from $0$ to $1$. The kernel is a Gaussian kernel of width $1.0$. The other three simulations are perturbations of this geodesic evolution in which we progressively increase the standard deviation of the white noise from $\epsilon = 0.9$ to $1.7$ and $3.5$. Remark that we have normalized the noise by dividing with the square root of $n \doteq 40$ as it will be suggested by extension to the infinite-dimensional case. It means that the total variance of the noise in the system is equal to $d\epsilon^2$ where $d$ is the dimension of the ambiant space ($d=2$).

The relative smoothness in time is evident in comparison to the simulations of the first order stochastic model.

\begin{figure}[htbp]
 \begin{minipage}{.45\linewidth}
  \centering	\includegraphics[width=7cm]{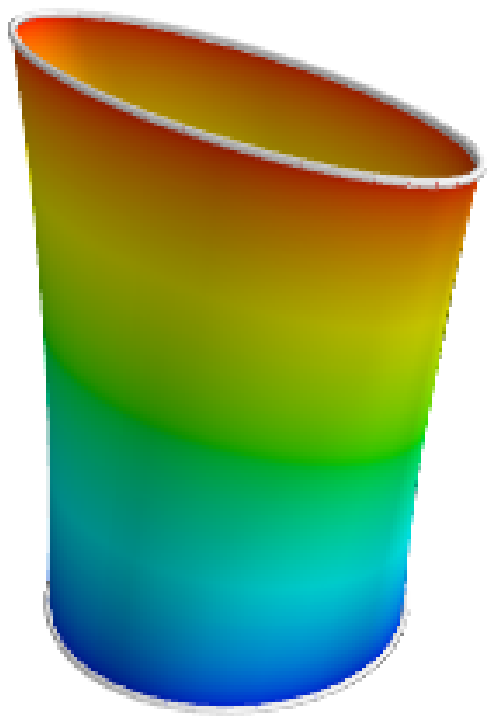}
 \caption{\footnotesize{Geodesic evolution - $40$ points on the unit circle.}}
 \label{GeodesicEvolution}
\end{minipage} \hfill
\begin{minipage}{.45\linewidth}
\centering\includegraphics[width=7cm]{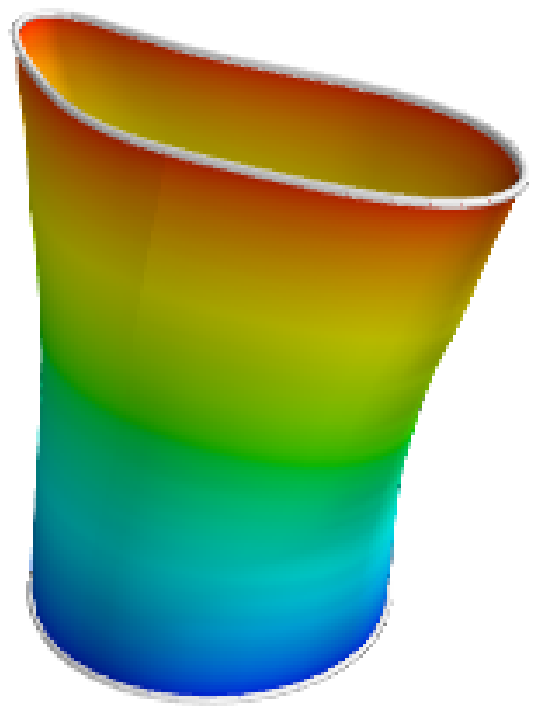}
\caption{\footnotesize{White noise perturbation of the geodesic - $\epsilon = 0.9$}}
  \label{GeodesicNoise0_9}
 \end{minipage} \hfill
\end{figure}

\begin{figure}[htbp]
 \begin{minipage}{.45\linewidth}
  \centering	\includegraphics[width=7cm]{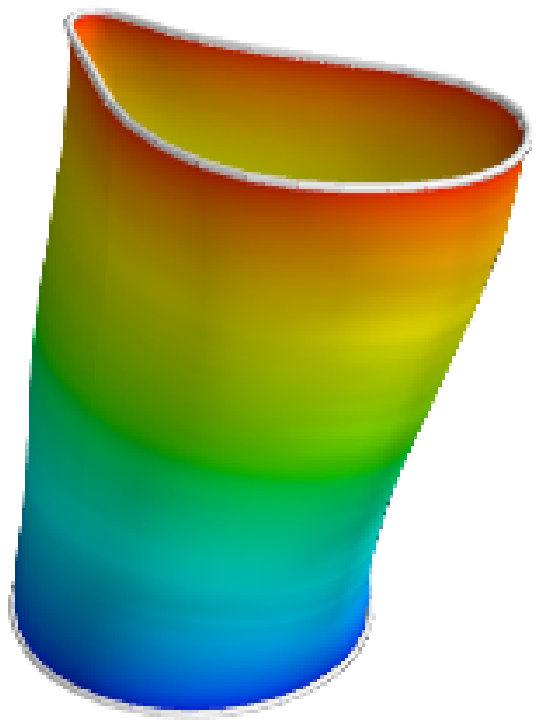}
 \caption{\footnotesize{Increasing the variance of the noise - $\epsilon = 1.7$}}
 \label{GeodesicNoise1_7}
\end{minipage} \hfill
\begin{minipage}{.45\linewidth}
\centering\includegraphics[width=7cm]{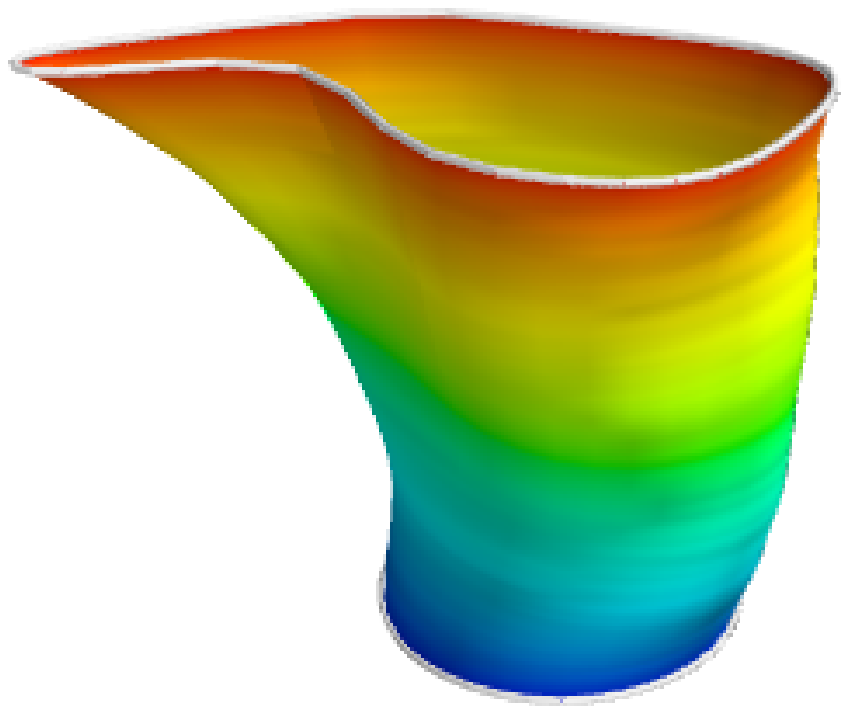}
\caption{\footnotesize{Increasing the variance of the noise continued - $\epsilon = 1.7$}}
  \label{GeodesicNoise3_5}
 \end{minipage} \hfill
\end{figure}

\subsection{Toward the stochastic extension} \label{prop_needed}

Back to the stochastic Hamiltonian system, a natural limit of the system \eqref{M1} appears when increasing the number of landmarks: roughly speaking the energy of the noise should be kept constant. Therefore the Brownian motion in the system \eqref{M1} could be interpreted as the projection of a cylindrical Brownian motion on $L^2(S_1,\ms{R}^2)$. A white noise on the particles can be extended to a white noise on the parameterization of the shape. However we would like to deal with more general noise than the white noise related to that particular parameterization, this is the reason why a general variance term will be studied.

\vspace{0.3cm}
\noindent
Let us first discuss this extension from a heuristic point of view.
\\
We give a short definition of the white noise on $H=L^2(S_1,\mu,\ms{R}^2)$ with $\mu$ the Lebesgue measure. We will come back to this definition later on.
\begin{defi}
Let $(B^i)_{i = 0}^{+\infty}$ be a family of independent real valued Brownian motions and $(e_i)_{i = 0}^{+\infty}$
an orthonormal and complete basis of $H$.
The process $B_t = \sum_{i=0}^{+\infty} B^i_t e_i$ is called a cylindrical Brownian motion on $H$.
\end{defi}

\noindent
Our approach leads to the following equations,
\begin{subequations} \label{syst2}
\begin{align}
&dp_t = -[p_t \int_{S_1} \partial_1 k(q_t,q_t(s))p_t(s)ds]dt + \ve dB_t\,, \label{eq1s2} \\
&dq_t = [\int_{S_1} k(q_t,q_t(s))p_t(s)ds]dt\,. \label{eq2s2}
\end{align}
\end{subequations}
\noindent
There are important issues to be discussed in the structure of this stochastic system. We study these Hamiltonian equations \eqref{syst2} with $q \in Q$ and $p \in P$ for $P$ and $Q$ some undefined Hilbert spaces. This study will provide us with some informations on suitable spaces to develop our approach.
\\
We want $Q$ to contain a large set of piecewise constant functions to account for the landmark case as described above for any number of particles .
\begin{property}
Piecewise constant (at least for a large range of partitions of $S_1$ in intervals) functions are contained in $Q$ such that the case of landmark can be treated within this space.
\end{property}
To properly define the term $\int_{S_1} k(.,q_t(s))p_t(s)ds$ in equation \eqref{eq2s2}, we can add these two following hypothesis
\begin{property}
$P$ and $Q$ are dual and we have the injections
\begin{equation*}
 Q \hookrightarrow L^2 \hookrightarrow P \, .
\end{equation*}
\end{property}
\begin{property} \label{stab_comp}
\item If $K$ is a smooth function, $f\in Q \mapsto K \circ f \in Q$ is locally Lipschitz.
\end{property}
Now the term $\int_{S_1} k(.,q_t(s))p_t(s)ds$ is well defined by $k_qp(.) \doteq (p,k(.,q))_{P,Q}$. If $k_qp(.)$ is sufficiently smooth, then the last property gives a sense to \eqref{eq2s2}:
$$ dq_t = k_qp \circ q \, dt \, .$$
Let us study equation \eqref{eq1s2} which can be rewritten as 
$$ dp_t = -p_t\, (dv_qp \circ q) \, dt + \ve dB_t\, ,$$
with $dv_qp(.) \doteq (p,\partial_1 k(.,q))_{P,Q}$. If this map is smooth enough, $dv_qp \circ q \in Q$ thanks to property \ref{stab_comp}. 
To give a sense to $ p (dv_qp \circ q) \in P$, we ask for the following
\begin{property} \label{stab_mult}
$Q$ is an algebra, the multiplication is continuous for the norm on $Q$.
\end{property}
Indeed if $p \in Q$ and $q_0 \in Q$, then we can define $p.q_0 \in P$ defined by:
$$ (p.q_0,q)_{P,Q} = (p,q_0q)_{P,Q}\, ,$$ the right-hand term being continuous w.r.t. $q$ since the product is continuous.
\\
Finally the noise term should belong to $P$ as follows,
\begin{property}
The paths of the cylindrical Brownian motion $t \rightarrow B_t$ lie almost surely in $C([0,T],P)$ for all $T>0$.
\end{property}
This last property ends to give a sense to equation \eqref{eq1s2}. 
\\
This set of conditions is a guide to get a proper space to prove the results. 
It is well known that $H^1(S_1)$ satisfies all these properties but it does not contain piecewise constant functions.
We will present in the next section a candidate for $(P,Q)$ that fulfills all the previous properties in any dimension (curves, surfaces, $\ldots$). 

\vspace{0.3cm}

Once the system is well-posed, the other issue to be discussed is the existence of solutions to this stochastic system on $P,Q$ and the convergence of the projections (landmark case) to the infinite-dimensional case. As we want to be slightly more general on the noise term, we will tackle the case when 
$$ \ve : P \times Q \mapsto L(P) \, $$
is \lip and bounded which would be a natural extension of proposition \ref{landmarkcase}. 

\section{The Hilbert spaces $F_s = Q$ and $F_{-s} = P$} \label{F_s}

In this section, we present the spaces $P,Q$ that verify all the properties of \ref{prop_needed}. At first sight, we could think about a Sobolev space on the Haar basis. Hopefully the properties we need would be verified. However, it is not convenient to work with Sobolev spaces on the Haar basis to prove the \lip property of the composition with a smooth function. This is our motivation to slightly modify this space by defining the space $F_s =Q$ for $s>0$ and $F_{-s} = P$ which are well suited to easily obtain the required properties of subsection \ref{prop_needed}. 

\vspace{0.3cm}
\noindent
Recall that $H=L^2(S_1,\ms{R})$, we consider the Haar orthonormal basis\index{Haar basis} with $\psi_0(x) = \chi_{[0,\frac{1}{2}[} - \chi_{[\frac{1}{2},1[}$ and 
$\psi_{n,k}(x) = 2^{\frac{n}{2}} \psi_0 \left(2^n(x-\frac{k}{2^n})\right)$ for $n\geq 1$ and $0 \leq k \leq 2^n-1$ and the constant function $\psi_{-1,0}:=1$.  
\\
We define the Haar coefficients of a function $f\in H$ by $$ f_{n,k} = \langle f, \psi_{n,k} \rangle_{H}\, ,$$
for $n\in [-1,+\infty[$ and $k \in A_n \doteq [0,2^n-1]$ if $n\geq 0$ and $A_{-1} = \{ 0 \}$.
\\
Let us start with a simple remark.

\begin{rem} \label{dualineq}
Let $g \in L^{\infty}(S_1,\ms{R})$ be a function, we have
$$ |g_{n,k}| \leq |g|_{\infty} |\psi_{n,k}|_{L^1} = 2^{-\frac{n}{2}}|g|_{\infty} \, .$$ 
\end{rem}

\begin{defi}
We define $H_{s} = \{ f \in H ; |f|_{H_{s}}^2 = \sum_{n=-1}^{\infty}\sum_{k \in A_n} 2^{ns}|f_{n,k}|^2 < +\infty \}$, with $s \geq 0$ a nonnegative real number. For $s<0$, we define $H_{s}$ as the dual of $H_{-s}$. 
\end{defi}

\noindent
We study some properties of an element in this Hilbert space.

\begin{prop} \label{continuity}
An element $f \in H_{s}$ for $s > 1$ is continuous in every $x \in [0,1]$ which is not a dyadic number, more precisely, 
if $a=\frac{2k+1}{2^{n+1}}$ with an integer $k$ such that $2k+1 \in [0,2^{n+1}-1]$ and $(x,y) \in B(a,\frac{1}{2^{n+1}})^2$, then with $C_s^2 = \sum_{i=0}^{+\infty} 2^{-i(s-1)} = \frac{2^{s-1}}{2^{s-1}-1}$
\begin{equation}
|f(x) - f(y)| \leq C_s \frac{2}{2^{n\frac{(s-1)}{2}}} |f|_{H_{s}} \,.
\end{equation}  
\end{prop}

\begin{proof}
We define for $x \in [0,1]$ and $n\geq -1$, 
$$ A_{x,n} = \{ k \in A_n \, | \,  x \in \text{Supp}(\psi_{n,k}) \}\,.$$
Remark that for each $n$ there exists one and only one $k$ in $A_{x,n}$. We will use this remark in the next inequalities. Also, the difference $|f(x) - f(y)|$ does not involve terms in the sequence that are constant on the ball $B(a,\frac{1}{2^{n+1}})$, thus we have with Cauchy-Schwarz inequality and using the remark \eqref{dualineq}, 

\begin{eqnarray*}
|f(x)-f(y)| &\leq & \sum_{l\geq n} \left( \sum_{k \in A_{x,l}} |f_{l,k}| 2^{\frac{l}{2}} +\sum_{k \in A_{y,l}} |f_{l,k}| 2^{\frac{l}{2}}\right)\,, \\
|f(x)-f(y)| &\leq & 2|f|_{H_s} \sqrt{\sum_{l\geq n} 2^{-l(s-1)}} \,, \\
|f(x) - f(y)|& \leq &C_s \frac{2}{2^{n\frac{(s-1)}{2}}} |f|_{H_{s}}\,,
\end{eqnarray*}
which is the result.
\end{proof}

\begin{rem} \label{domin}
This proof gives also that the sup norm is bounded by the $H_s$ norm for $s>1$:
$$ |f|_{\infty} \leq C_s |f|_{H_{s}} \, .$$
\end{rem}

Now, we introduce the suitable Hilbert space $F_s$. 

\begin{defi}
We define the Hilbert space for $s\geq 0$, $$ F_s = \{ f \in H \, | |f|^2 = \int_0^1 f^2 \,dx \, + \, \sum_{n,k} 2^{ns-1} \int_{I_{n,k}} |f(x+ \frac{1}{2^{n+1}} ) - f(x)|^2 dx < \infty\}, $$ with $I_{n,k} = [\frac{k}{2^{n}},\frac{k}{2^{n}}+\frac{1}{2^{n+1}}]$. Its dual is denoted by $H_{-s}$.
\end{defi}

\noindent
We have the following inclusion
\begin{prop}
We have the inclusion $F_s \subset H_s$ and if $s>1$ and
$$|f|_{H_s} \leq |f|_{F_s} \, .$$
\end{prop}

\begin{proof}
To see this fact, we apply the Cauchy-Schwarz inequality
\begin{eqnarray*}
f_{n,k}^2 = 2^{n}(\int_{I_{n,k}} f(x+ \frac{1}{2^{n+1}} ) - f(x) dx)^2 \,, \\ 
f_{n,k}^2 \leq \frac{1}{2} \int_{I_{n,k}} |f(x+ \frac{1}{2^{n+1}} ) - f(x)|^2 dx\,.
\end{eqnarray*}
\noindent
Moreover, $$ (\int_{S_1} f(x) dx)^2 \leq \int_{S_1} f(x)^2 dx\,,$$
so that we have $|f|_{H_s} \leq |f|_{F_s}$. 
\end{proof}

\noindent
We want our space to be big enough to contain usual functions.
In the following, we prove that $F_s$ contains Lipschitz functions for $s<2$ and also, if $s<2$:
$$ \text{Lip}_{dyad}(S_1) \doteq \{ f \in L^2(S_1,\ms{R}) | \exists n , \, f_{| I_{n,k}} \in \text{Lip}(I_{n,k},\ms{R}) \, \forall k \in [0,2^n-1] \} \subset F_s\,. $$
This fact is important since it means that we can deal with a wide range of shapes in this space.

\begin{prop} \label{contains_lip}
If $s<2$, $F_s$ contains piecewise \lip functions,
$$\text{Lip}_{dyad}(S_1)\subset F_{s}.$$
\end{prop}

\begin{proof}
Let $f \in \text{Lip}(S_1,\ms{R})$ be a Lipschitz function and $M$ be its \lip constant.
Then we have
\begin{equation*}
2^{ns-1} \int_{I_{n,k}} |f(x+ \frac{1}{2^{n+1}} ) - f(x)|^2 dx \leq \frac{M 2^{ns-1}}{2^{3n+3}} \,\\ 
\end{equation*}
so that if $s<2$, $f \in F_{s}$.
\end{proof}
\noindent
The following proposition is needed to ensure the stability of our stochastic system.
For example, if $G$ is \lip and bounded, the growth of $G \circ f$ is linear:
\begin{prop} \label{lipstability}
If $s>1$, $G$ a real \lip function and $f \in F_s$, then $G \circ f \in F_s$ and also, 
\begin{equation}
\begin{aligned}
|G \circ f|_{F_s}& \leq \text{Lip}(G) |f|_{F_s} + |G \circ f|_{L^2}\,, \\
|G \circ f|_{F_s}& \leq \text{Lip}(G) |f|_{F_s} + |G(0)|\,.
\end{aligned}
\end{equation}
with $\text{Lip}(G)$ the \lip constant for $G$.  
\end{prop}
\begin{proof}
Applying the \lip property we have, $$\int_{I_{n,k}} |G \circ f(x+\frac{1}{2^{n+1}}) - G \circ f(x)|^2 \, dx \leq \int_{I_{n,k}} \text{Lip}(G)^2 |f(x+\frac{1}{2^{n+1}}) - f(x)|^2 \, dx \, ,$$ 
then,
\begin{equation} \label{intermed_lip_comp}
|G \circ f|_{F_s}^2 - |G \circ f|_{L^2}^2 \leq \text{Lip}(G)^2 (|f|_{F_s}^2 - |f|_{L^2}^2) \, .
\end{equation} 
Since for any couple $(a,b)$ of nonnegative real numbers $a^2+b^2 \leq (a+b)^2$, we have the first inequality. 
\\
The second inequality is the application of inequality \eqref{intermed_lip_comp} to the function $g = G \circ f - G(0)$ using the inequality $|g|_{L^2} \leq \text{Lip}(G)|f|_{L^2}$. 
\end{proof}
\noindent
We need to go further by proving that the composition is locally \lip if $G'$ is Lipschitz.

\begin{prop} \label{loclip}
If $s>1$, $G$ a real $C^1$ function such that $G'$ is locally \lip then, for any $r>0$ there exists $M>0$ such that 
$$|G \circ f_1 - G \circ f_2| \leq M|f_1-f_2| \, ,$$
if $(f_1,f_2) \in B(0,r)^2$. 
If $G'$ and $G''$ are bounded then the \lip constant $M$ has linear growth,
$$M \leq \sqrt{2}(|G'|_{\infty} + 3C_sr |G''|_{\infty})\,.$$
\end{prop}

\begin{proof}
For notation convenience, we will denote by for $i = 1,2$ and $\delta>0$, $$\Delta f_i(x) = f_i(x+\delta) - f_i(x)\,.$$
We will control the quantity $$\int_{I_{n,k}} |(G\circ f_1 -G \circ f_2)(x+ \frac{1}{2^{n+1}} ) - (G\circ f_1 -G \circ f_2)(x)|^2 dx\,.$$
We will divide by $\Delta f_1$, it is permitted in this situation, since we can extend the definition with the equation \eqref{division}. Let $a \leq b$ be two real valued numbers. We have, with $\mu$ the Lebesgue measure
\begin{eqnarray} \label{decomposition_loclip}
\int_a^b |\Delta (G\circ f_1 -G \circ f_2) |^2 d\mu \leq 2 \int_a^b |\frac{\Delta (G \circ f_1)}{\Delta f_1} \Delta f_1 - \frac{\Delta (G \circ f_2)}{\Delta f_2} \Delta f_1|^2 d\mu \\ \nonumber
+ 2\int_a^b |\frac{\Delta (G \circ f_2)}{\Delta f_2} \Delta f_1 - \frac{\Delta (G \circ f_2)}{\Delta f_2} \Delta f_2|^2 d\mu\,.
\end{eqnarray}
Observe that the second term can be bounded easily:
$$ \int_a^b |\frac{\Delta (G \circ f_2)}{\Delta f_2} \Delta f_1 - \frac{\Delta (G \circ f_2)}{\Delta f_2} \Delta f_2|^2 d\mu \leq \int_a^b (\sup_{|y| \leq C_s r} |G'(y)|)^2 \, |\Delta f_1 - \Delta f_2|^2 d\mu\,,$$
and we have the \lip property on this term. Now we bound the first term. Remark that
\begin{equation} \label{division}
\frac{\Delta (G \circ f_i)}{\Delta f_i}(x) = \int_0^1 G'(t \Delta f_i(x) + f_i(x)) dt \, , 
\end{equation}
we get,
\begin{equation*}
|\frac{\Delta (G \circ f_1)}{\Delta f_1}  - \frac{\Delta (G \circ f_2)}{\Delta f_2}|_{\infty} \leq \int_0^1 \sup_{|y| \leq 3C_sr} |G''(y)| |t\Delta(f_1-f_2) + f_1-f_2|_{\infty} \,  dt\,. 
\end{equation*}
Since for $|t| \leq 1$, $|t\Delta(f_1-f_2) + f_1-f_2|_{\infty}  \leq 3|f_1-f_2|_{\infty} \leq 3C_s|f_1-f_2|_{F_s}$,
we obtain,
\begin{equation*}
|\frac{\Delta (G \circ f_1)}{\Delta f_1}  - \frac{\Delta (G \circ f_2)}{\Delta f_2}|_{\infty} \leq \sup_{|y| \leq 3C_sr} |G''(y)| 3C_s|f_1-f_2|_{F_s} \,. 
\end{equation*}
Back to the inequality \eqref{decomposition_loclip}, we obtain
\begin{eqnarray*}
\int_a^b |\Delta (G\circ f_1 -G \circ f_2) |^2 dx \leq 2  \int_a^b (\sup_{|y| \leq C_s r} |G'(y)|)^2 \, |\Delta (f_1 - f_2)|^2 d\mu + \\
2\int_a^b \left( \sup_{|y| \leq 3C_sr} |G''(y)| 3C_s|f_1-f_2|_{F_s} \right)^2 |\Delta f_1|^2 \, d\mu\,.
\end{eqnarray*}
Remark also that on the $L^2$ norm, applying the \lip property,
\begin{equation*}
|G \circ f_1 - G \circ f_2|_{L^2}^2 \leq \sup_{|y| \leq C_s r} |G'(y)|^2 |f_1-f_2|^2_{L^2}\,.
\end{equation*}
We get the result,
\begin{equation*}
|G \circ f_1 - G \circ f_2|^2_{F_s} \leq   2 \left( (\sup_{|y| \leq C_s r} |G'(y)|)^2 + (\sup_{|y| \leq 3C_sr} |G''(y)| 3C_s )^2 |f_1|_{F_s}^2 \right)  |f_1-f_2|_{F_s}^2\,.
\end{equation*}
To be more precise in the proposition, we obtain the following inequality on the \lip constant of the composition on every $F_s$ ball of radius $r>0$:
\begin{equation} \label{controllip}
M \leq \sqrt{2 (\sup_{|y| \leq C_s r} |G'(y)|)^2 + 2(\sup_{|y| \leq 3C_sr} |G''(y)| 3C_s r)^2 }\,.
\end{equation}
The linear growth of the \lip constant is the direct application of this inequality.
\end{proof}

\begin{prop} \label{loclip2}
Let $K: \ms{R}^j \mapsto \ms{R}$ be a $C^1(\ms{R}^n,\ms{R})$ function with $K'$ locally Lipschitz, $s>1$, then
$(f_1,\ldots,f_j) \in (F_s)^j \mapsto K(f_1,\ldots,f_j) \in F_s$ is Lipschitz on every bounded ball.
\end{prop}

We do not give the details of the proof of this proposition since it is a particular case of a generalization of this proposition which will be stated in proposition \ref{composition2D}.
A direct consequence of this proposition is that $F_s$ is an algebra. In the next proposition, we give an explicit bound for the continuity of the multiplication. Even if it is a direct application of the last proposition, we can give a better bound.

\begin{prop}
The product in $F_s$ is continuous for $s>1$ and,
$$ |fg|_{F_s} \leq 2 \, C_s \, |f|_{F_s}|g|_{F_s}.$$
\end{prop}

\begin{proof}
First we bound the $L^2$ norm
\begin{eqnarray*}
 |fg|_{L^2}^2 \leq |f|_{\infty}^2|g|_{L^2}^2\,,\\
 |fg|_{L^2}^2 \leq C_s^2 |f|_{F_s}^2 |g|_{L^2}^2\,.
\end{eqnarray*}
Now with the inequality $$|\Delta (fg)|^2 \leq 2(|g|_{\infty}^2 |\Delta f|^2 + |f|_{\infty}^2 |\Delta g|^2)\,,$$  
we obtain the result $ |fg|_{F_s}^2 \leq 4C_s^2 |f|_{F_s}^2 |g|_{F_s}^2$.
\end{proof}

The first natural generalization in two dimensions of $H_s$ is the tensor product $H_s \otimes H_{s'}$. We could do the same for $F_s$ but we prefer a slightly different definition of a space $F_{s,s'}$ (Although it turns to define exactly the tensor product $F_s \otimes F_{s'}$). We will take advantage of this definition to extend the composition property.

\begin{defi} \label{iter_F_s}
Let $(s,s')$ be two positive real numbers, the space $F_{s,s'} \subset L^2(T_2)$ is defined by 
$$ F_{s,s'} = F_s(S_1,F_{s'}),$$
in the following sense with $I_{n,k} = [\frac{k}{2^{n}},\frac{k}{2^{n}}+\frac{1}{2^{n+1}}]$,
$$|f|_{F_{s,s'}}^2 = \int_{S_1} |f|_{F_{s'}}^2 \,dx \, + \, \sum_{n,k} 2^{ns-1} \int_{I_{n,k}} |f(x+ \frac{1}{2^{n+1}} )(.) - f(x)(.)|_{F_{s'}}^2 dx < \infty\,.$$
\end{defi}

\begin{rem} \label{norme_tens}
This definition can be rewritten in a more explicit form,
\begin{eqnarray} \label{norme}
F_{s,s'}(T) &=& \{ f \in L^2(T) | \, |f|_{F_{s,s'}}^2 = \int_{S_1} |f(x,.)|_{F_{s'}}^2 \,dx \, + \int_{S_1} |f(.,y)|_{F_{s}}^2 \,dy \, + \nonumber \\ 
&&\sum_{n,k,n',k'} 2^{ns+n's'-2} \int_{I_{n,k}} \int_{I_{n',k'}} |\Delta_{2,n'}(\Delta_{1,n} f)|^2 dx\,dy \, < \infty      \}\,, 
\end{eqnarray}
with $\Delta_{1,n} f = f(x+ \frac{1}{2^{n+1}},y ) - f(x,y)$ and $\Delta_{2,n} f = f(x,y+ \frac{1}{2^{n+1}}) - f(x,y)$.
\end{rem}

\vspace{0.3cm}
\noindent
In what follows, we will denote $c_{(n,k),(n',k')}^2 := \int_{I_{n,k}} \int_{I_{n',k'}} |\Delta_{2,n'}(\Delta_{1,n} f)|^2 dx\,dy$.

\vspace{0.3cm}
\noindent
As in the one-dimensional case, the injection $F_{s,s'} \hookrightarrow H_{s,s'}$ would be again a straightforward application of Cauchy-Schwarz inequality. However we provide a different proof using the following proposition:

\begin{prop} \label{tens=comp}
As defined in \ref{iter_F_s}, we have
$$F_{s,s'} = F_s \otimes F_{s'} \, ,$$
and as a consequence,
$F_{s,s'} \hookrightarrow H_{s,s'}$.
\end{prop}
\noindent
Remark that the identification is here allowed since each of the two spaces are $L^2$ subspaces.
\begin{proof}
Let $(n,k),(n',k')$ be two couples of integers, then $(\phi_{n,k} \otimes \psi_{n',k'})_{(n,k),(n',k')}$ is an orthonormal Hilbert basis in $F_s \otimes F_{s'}$ by the definition. From the remark \ref{norme_tens}, this is also true in $F_{s,s'}$. 
\\
If $u \otimes v \in F_s \otimes F_{s'}$, then $$|u \otimes v|_{H_{s,s'}} = |u|_{H_s} |v|_{H_{s'}} \leq |u|_{F_s} |v|_{F_{s'}} \, ,$$
and the second assertion is verified.
\end{proof}

\noindent
Now we prove that if $f$ is sufficiently smooth then $f$ belongs to $F_{s,s'}$. It requires to control one more derivative than in the one-dimensional case. 

\begin{prop} \label{bigenough}
Let $f$ be a $C^1$ function such that $y \rightarrow \partial_{1}f(x,y)$ is Lipschitz uniformly in the first variable $x$, then $f \in F_{s,s'}$ for $s<2$ and $s'<2$.
\end{prop}

\begin{proof}
We will denote by $\text{Lip}_{1,2}$ the \lip constant in the second variable of $\partial_1 f$.
We need to estimate the integrals detailed in \eqref{norme}, we will denote by $c_{(n,k),(n',k')}$ each integral. Using the \lip property, we have
\begin{equation*} 
|\Delta_{2,n'}(\Delta_{1,n} f)| \leq  \int_x^{x+\frac{1}{2^{n+1}}} |\partial_{1}f(u,y+\frac{1}{2^{n'+1}}) - \partial_{1}f(u,y)| \, du 
\leq \text{Lip}_{1,2} 2^{-(n'+n+2)} \, .
\end{equation*}
Then, we have
\begin{equation*}
c_{(n,k),(n',k')} = \int_{I_{n,k}} \int_{I_{n',k'}} |\Delta_{2,n'}(\Delta_{1,n} f)|^2 dx\,dy \leq \text{Lip}_{1,2}^2 \, 2^{-3(n'+n+2)}.
\end{equation*}
\noindent
Summing up on $(k,k')$, we get that 
\begin{equation} \label{convergence}
\begin{split}
\sum_{(n,k),(n',k')} 2^{ns+n's'-2} c^2_{(n,k),(n',k')} & \leq \sum_{(n,k),(n',k')} \text{Lip}^2_{1,2} \,  2^{n(s-3) + n'(s'-3) -8}, \nonumber \\
 & \leq \sum_{n,n'} \text{Lip}^2_{1,2} \, 2^{n(s-2) + n'(s'-2) -8}. 
\end{split}
\end{equation}

\noindent
In the equation \eqref{convergence}, the right-hand side converges if $s<2$ and $s'<2$. To conclude the proof, the first two terms in the $F_{s,s'}$ norm according to the formula \eqref{norme} are well defined thanks to the proposition \ref{contains_lip} which requires $s<2$ and $s'<2$.
\end{proof}

\begin{rem}
Again, any function for which there exists a finite dyadic partition of $T$ such that the restriction of $f$ on each domain of the partition satisfies the condition in the proposition \ref{bigenough} belongs to $F_{s,s'}$. 
\end{rem}

\vspace{0.2cm}
\noindent
Next, we will state the relevant inequalities to prove the needed properties. At this point, it is worthwhile to generalize a little the spaces we introduced in order to extend our results to any dimension. Observe that $F_s$ is a Hilbert algebra of functions. A way to generalize our result is to study
$F_s(S_1,F)$ where $F$ is a Hilbert space of functions which is also an algebra. Although we could be more general, we will assume further that $F \subset L^2(T_n)$ where $T_n$ is the $n$-dimensional torus: $T_n \doteq \ms{R}^n / \ms{Z}^n$. As in the previous definition,

\begin{defi}
 Let $s>0$ be a positive real number, the space $F_{s}(F)$ is by 
$$ F_{s}(F) = F_s(S_1,F) \, ,$$
in the following sense: 
\\
with $I_{n,k} = [\frac{k}{2^{n}},\frac{k}{2^{n}}+\frac{1}{2^{n+1}}]$, a function $f \in L^2(T_{n+1})$ belongs to $F_s(F)$ if 
$$|f|_{F_{s}(F)}^2 := \int_{S_1} |f(x)|_{F}^2 \,dx \, + \, \sum_{n,k} 2^{ns-1} \int_{I_{n,k}} |f(x+ \frac{1}{2^{n+1}} ) - f(x)|_{F}^2 dx < \infty\,.$$
We have
\begin{equation} \label{egalite_tens_comp}
F_{s}(F) = F_s \otimes F \,.
\end{equation}
\end{defi}
\noindent
The last equation \eqref{egalite_tens_comp} uses the same argument as in proposition \ref{tens=comp}.

\begin{prop} \label{normes}
Let us assume that $F$ is a RKHS on a space $X$ and that there exists a constant $C_F$ such that $\sup_{x \in X} |f(x)| \leq C_F|f|_{F}$.
If $s>1$ then the following inequalities hold for $f \in F_{s}(F) = F_s \otimes F$,
\begin{align*}
& \sup_{s \in S_1} |f(s)|_F \leq C_s |f|_{F_s(F)} \, \\
& \sup_{s \in S_1} |\Delta_1 f|_{F} \leq C_{s} |\Delta_1 f|_{F_s(F)} \, \\
& \sup_{(s,x) \in S_1 \times X}|f(s,x)|_{\infty} \leq C_s \, C_{F} |f|_{F_{s}(F)} \,, 
\end{align*}
where $\Delta_1$ is a difference operator defined for $\delta>0$ as $(\Delta_1 g )(x) \doteq g(x+\delta) - g(x)$.
Moreover proposition \ref{continuity} is also verified.
\end{prop}

\begin{proof}
If $u \in F_s \otimes F$, then $u = \sum_{n,m \in \ms{N}^2} \alpha_{n,m} e_n \otimes f_n$, with $(e_n)_{n \in \ms{N}}$ and $(f_n)_{n \in \ms{N}}$ Hilbert basis respectively for $F_s$ and $F$. By definition, $\sum_{n,m \in \ms{N}^2} \alpha_{n,m}^2 < \infty$. Hence, we have,
$$ u= \sum_{m \in \ms{N}} (\sum_{n \in \ms{N}} \alpha_{n,m} e_n) \otimes f_m  = \sum_{m \in \ms{N}} E_m \otimes f_m \,.$$
Then, we can apply the evaluation at point $z \in S_1$ since $F_s$ is also a RKHS. We denote by $C_z$ the norm of the evaluation at point $z \in S_1$:
the sequence $\sum_{m =0}^N E_m(s) f_m$ is a Cauchy sequence in $F$ since 
$$|\sum_{n=p}^qE_n(z)f_n|_F^2=\sum_{n=p}^q E_n(z)^2\stackrel{\text{rkhs prop.}}{\leq} C_{z}^2\sum_{n=p}^q |E_n|_{F_s}^2\, .$$
As a consequence, the evaluation at point $z \in S_1$ is well defined and it makes the space $F_s \otimes F$ a RKHS on $S_1$ with values in $F$.
Furthermore, as $C_{z} \leq C_s$, we have:
$$ \sup_{z \in S_1} |f(z)|_H \leq C_s |f|_{F_s(F)} \, .$$
The second inequality is the application of the first one to $\Delta_1f$ and the last one just uses the assumption on the RKHS $F$.
\end{proof}

\noindent
Now, we can easily generalize the work done in one dimension. 

\begin{prop} \label{composition2D}
Let $F \subset L^2(T_n,\ms{R}^k)$ for $k \geq 1$ be a RKHS algebra with a continuous injection in $L^{\infty}(T_n,\ms{R}^k)$. Assume that the left composition with an element $H \in C^l(\ms{R}^k,\ms{R}^k)$
$$ F \ni g\mapsto H \circ g \in F$$ is \lip on every ball $B(0,r)$ of constant $C_H(r)$. 
Then, 
\begin{itemize}
\item if $G \in C^{l+1}(\ms{R}^k,\ms{R}^k)$, the composition
$$ F_s(F) \ni f \mapsto G \circ f \in F_s(F)\,, $$
is \lip on every ball,
\item with the additional assumption that the left composition with $G'$ and $G''$ in $F$ are locally \lip such that there exists a polynomial real function $P$ verifying $\max(C_{G'}(r),C_{G}(r)) \leq P(r)$ for $r>0$, then there exists a constant depending on $F$ and $s$, $a \in \ms{R}^+$ such that the \lip constant $C_{G,F_s(F)}$ for the left composition with $G$ on $F_s(F)$
$$C_{G,F_s(F)}(r) \leq  arP(r)\,.$$ 
\end{itemize}
\end{prop}

\begin{proof}
We first need to prove that if $f \in F_s(F)$ then $G \circ f \in F_s(F)$. 
With the proposition \ref{normes}, we have that 
\begin{equation} \label{localineq}
\sup_{x \in S_1} |f(x)|_H \leq C_s |f|_{F_s(F)}.
\end{equation}
 Then, we obtain for the first term in the norm,
$$ \int_{S_1} |G \circ f - G(0)|_{F}^2 d\mu \leq C_s^2 C_G^2(C_s |f|_{F_s(F)}) |f|^2_{F_s(F)}\,. $$
For the term involving the difference, we need to introduce again the formula,
\begin{equation} \label{division2}
\Delta_x (G \circ f) = \, (\int_0^1 G'(t \Delta_x f + f)(\Delta_x f) dt) \, , 
\end{equation}
which is now allowed since $G'$ is $C^{l}$. The formula \eqref{division2} uses the fact that $F \subset L^2(T_n,\ms{R}^k)$ to give a sense to the composition.
As $F$ is an algebra, we have $G'(t \Delta_x f + f)\Delta_x f \in F$. 
Obviously we have also $ t \Delta_x f + f \in B_F(0,3r_0)$ for $|t| \leq 1$ and $r_0 = C_s|f|_{F_s(F)}$. With the inequality \eqref{localineq}, we get
\begin{equation} \label{localineq2}
|G'(t \Delta_x f + f)\Delta_x f|_{F} \leq 3 M r_0 |\Delta_x f|_{F} C_{G'}(3r_0)\,,
\end{equation}
with $M$ the constant associated with the continuity of the multiplication in $F$:
$$ \forall (f,g) \in F^2 \, , \, |fg|_F \leq M|f|_F|g|_F \, .$$
Remark that $G'(t \Delta_x f + f)$ can be seen as a matrix valued function. We use the matrix norm implied by the Euclidean norm on $\ms{R}^k$.
The inequality \eqref{localineq2} directly proves that $G \circ f \in F_s(F)$ with in addition:
\begin{equation*}
|G \circ f - G(0)|_{F_s(F)}^2 \leq [\max( 3Mr_0C_{G'}(3r_0) , C_sC_G(r_0) )]^2 \, |f|^2_{F_s(F)}\,.
\end{equation*}

\vspace{0.3cm}
\noindent
We now prove the \lip property:
\\
let $f_1$ and $f_2$ be two elements in $F_{s}(F)^2$ with $\max(|f_1|_{F_{s}(F)},|f_2|_{F_{s}(F)}) \leq r_1$. With the \lip property of the composition on $F$ we have with $r_2 = C_s r_1$,
\begin{equation} \label{localfirstlip}
\int_{S_1} |G \circ f_1(x) - G \circ f_2(x)|_{F}^2 dx \leq \int_{S_1} C_G^2(r_2)|f_1(x) - f_2(x)|_{F}^2 dx\,.
\end{equation}
For the remaining terms, we use again the formula \eqref{division2}:
\begin{equation} \label{split0}
\begin{split}
|\Delta_x (G \circ f_1 - G \circ f_2)|_{F}^2 & \leq 2M^2 [ |\Delta_x (f_1-f_2)|_{F}^2 (\int_0^1 |G'(t \Delta_x f_1 + f_1)|_{F} \, dt)^2 \\
& \quad + |\Delta_x f_2|_{F}^2 (\int_0^1 |G'(t \Delta_x f_1 + f_1) - G'(t \Delta_x f_2 + f_2)|_{F} \, dt)^2  ]\,.
\end{split}
\end{equation}
The last term of inequality \eqref{split0} can be bounded as follows,
\begin{equation}
\int_0^1 |G'(t \Delta_x f_1 + f_1) - G'(t \Delta_x f_2 + f_2)|_{F} \, dt \leq 3 C_s |f_1-f_2|_{F_s(F)} C_{G'}(3r_2)\,. 
\end{equation}
We then obtain,
\begin{equation} \label{splitFinal}
\begin{split}
|\Delta_x (G \circ f_1 - G \circ f_2)|_{F}^2 & \leq 2M^2 [ |\Delta_x (f_1-f_2)|_{F}^2 (3 r_2 C_{G'}(3r_2))^2 \\
& \quad + |\Delta_x f_2|_{F}^2 (3 C_s |f_1-f_2|_{F_s(F)} C_{G'}(3r_2))^2 ]\,.
\end{split}
\end{equation}
For notation convenience, we define $K(r_1) :=  3 r_2 C_{G'}(3r_2)$ and we finally get, combining equations \eqref{localfirstlip} and \eqref{splitFinal}
\begin{equation} \label{localfinal}
|G \circ f_1 - G \circ f_2|_{F_s(F)}^2 \leq \max (4M^2 K^2(r_1),C_G^2(r_2)) |f_1-f_2|^2_{F_s(F)} \, .
\end{equation}
This inequality implies directly the last item in the proposition.
\end{proof}

\noindent
We have presented all the material necessary to generalize easily our results. We generalize $F_{s}$ in dimension $n\geq 3$ as it is already done in one and two dimensions.
\begin{defi}
We define by recurrence $F_{s}$ where $s=(s_1,\ldots,s_n) \in \ms{R}_+^n$ by
for $n \geq 3$, 
$$ F_{s_1}(S_1,F_{(s_2, \ldots, s_n)}) = F_{s_1} \otimes \ldots \otimes F_{s_n} \,.$$
We denote its dual $F_{-s}$, $s_* = \min_{i \in [1,n]} s_i$ and $s^* = \max_{i \in [1,n]} s_i$.
\end{defi}
\noindent
To sum up our work to this point, we have defined a RKHS algebra $F_s$ for a multi-index $s$ which is stable under the composition with smooth functions. The continuity of the product is detailed in appendix with proposition \ref{tensor_product_algebras} (it is also a byproduct of the previous result on the composition with smooth functions in proposition \ref{composition2D}). We will now prove that $F_{s}$ is big enough. Provided that linear functions are in $F_s$, the proposition of the composition \ref{composition2D} answers this question. We can have a better result:

\begin{prop} \label{bigenough2}
If $s^*<2$ and $f \in C^n(T_n,\ms{R}^k)$ then $f \in H_s$ and $|f|_{H_s} \leq c_{n}|f|_{n,\infty}$.
\end{prop}

\begin{proof}
By recurrence, this true for $n=1$.
With the inequality, 
\begin{equation}
\begin{aligned}
|f|^2_{F_s} & \leq \int_{S_1} c_{n-1}^2|f(x)|_{n-1,\infty}^2 dx + \sum_{n,k} 2^{ns-1}\int_{I_{n,k}} (\int_x^{x+\frac{1}{2^{n+1}}} c_{n-1} |\partial_1 f|_{n-1,\infty} dx)^2\,, \\
|f|^2_{F_s} & \leq c_{n-1}^2 |f|_{n,\infty}^2 (1 + \sum_{n=0}^{\infty} 2^{n(s-2)-4})\,.
\end{aligned}
\end{equation}

\noindent
We have the result with $c_n^2 = c_{n-1}^2 (1 + \sum_{n=0}^{\infty} 2^{n(s-2)-4})$.
\end{proof}
As a direct application of this proposition, we get
\begin{defi}
If $n \geq 2$, a dyadic partition of $T_n$ is a product of a dyadic partitions in one dimension.
\end{defi}
\begin{prop}
If $\max s<2$ and $f \in L^2(T_n,\ms{R}^k)$ such that there exists a dyadic partition on which the restriction of $f$ is $C^n$ then $f \in H_s$.
\end{prop}

\noindent
The space of functions such that the restriction is $C^p$ on a dyadic will be denoted $C_{\text{dyad}}^{p}(T_n,\ms{R}^k)$. In the next section, we will present the cylindrical Brownian motion and we will prove that almost surely its trajectories are continuous paths in $H_{-s}$, therefore in $F_{-s}$. To sum up the properties of $F_s$,
\begin{theo} \label{resume_P_Q}
The Hilbert space $F_s \subset L^2(T_n,\ms{R}^{k})$ satisfies the following properties
\begin{itemize}
\item the left composition with a function $G \in C^{n+1}(\ms{R}^k,\ms{R}^k)$ is locally Lipschitz,
\item if $s_*>1$ then $F_s$ is an algebra with continuous product,
\item the cylindrical Brownian motion defines a continuous random process in $F_{-s}$,
\item if $s^*<2$ then $C_{\text{dyad}}^{p}(T_n,\ms{R}^k) \subset F_s$.
\end{itemize}
\end{theo}
\noindent
We can now deduce an important property to prove the existence for all time of the SDE solutions.

\begin{prop} \label{prop19}
If the kernel $k$ has continuous derivatives,
$$ \frac{\partial^{l+n} k(x,y)}{\partial_l x \partial_n y} \, l,n \in [1,m+2] \, ,$$
a couple $(p,q) \in F_{-s} \times F_{s}$ defines an element of $V$ by $k_qp(x) = \langle k(x,q),p \rangle_{F_s \times F_{-s}}$, 
with $s=(s_1,\ldots,s_m)$. 
Moreover, the following mappings are \lip
\begin{gather}
F_{-s} \times F_{s} \ni (p,q) \mapsto k_qp \circ q \in F_s \, ,\\
F_{-s} \times F_{s} \ni (p,q) \mapsto \langle p, (\partial_x k_qp) \circ q \rangle_{F_{-s} \times F_{s}} \in F_{-s} \, .
\end{gather}
\end{prop}

\begin{proof}
First, for any $x$, $k_qp(x)$ is well defined: since $y \mapsto k(x,y)$ is $C^{m+2}$, we apply the Theorem \ref{resume_P_Q} to get $k(x,q) \in F_s$. Hence, $k_qp(x)$ is well defined. Our goal is to prove that $k_qp(x)$ is $C^{m+2}$ and $\partial_x k_qp$ is $C^{m+1}$. Thus we would obtain that $k_qp \in V$. The composition with a $C^{m+1}$ function being locally \lip on $F_{s}$, the results will follow.

\vspace{0.2cm}
\noindent
To prove the continuity of $x \mapsto k_qp(x)$, we just need the weak convergence in $F_s$: $$k(x_n,q(.)) \rightharpoonup_{n \mapsto \infty} k(x,q(.)) $$ 
when $\lim_{n \mapsto \infty} x_n = x$. 
\\
We first prove that $k(x_n,q(.)) \rightarrow_{L^{\infty}} k(x,q(.))$: thanks to the injection $F_s \hookrightarrow L^{\infty}(T_m,\ms{R}^d)$, $|q|_{\infty} \leq r_0$. Since $k$ is continuous, it is uniformly continuous on $W \times \overline{B(0,r_0)}$ with $W$ a compact neighborhood of $x$. 
\\
Then, for any $\ve>0$, there exists $\delta>0$ such that if $|x_n-x|\leq \delta$, we have $|k(x_n,y)-k(x,y)| \leq \ve$ and as a consequence $|k(x_n,q)-k(x,q)|_{\infty} \leq \ve$.

\vspace{0.2cm}

We now prove that $k(x_n,q)$ is bounded in $F_s$: as $\partial^n_y k(x,y)$ for $n\in [1,m+1]$ is bounded on any compact set, we get that $k(x_n,q(.)) \in F_s$ is bounded in $F_s$. 
\\
Since $L^1 \subset L^{\infty}(T_m)' \subset F_{-s}$ (thanks to $F_s \hookrightarrow L^{\infty}(T_m)$) every weak subsequence of $k(x_n,q(.))$ converges to $k(x,q(.))$. Then, $\lim_{n \mapsto \infty} k_qp(x_n) = k_qp(x)$. Therefore $k_qp$ is continuous. By the same proof, $k_qp$ is a $C^1$ function: 
\\
as $k$ is $C^{m+1}$, we apply the same argument to $\frac{k_qp(x+t_nv)-k_qp(x)}{t_n} - \partial_1 k_qp(x)(v)$ for $v \in \ms{R}^d$ and $t\neq0$. We have,
$$\frac{k_qp(x+t_nv)-k_qp(x)}{t_n} - \partial_1 k_qp(x)(v) = \int_0^1 \partial_1 k_qp(x+st_nv,q)(v)p - \partial_1 k_qp(x)(v) ds \, .$$
The sequence $\int_0^1 \partial_1 k_qp(x+st_nv,q(.))(v) - \partial_1 k(x,q(.))(v) ds$ converges in $L^{\infty}$ to $0$ by uniform continuity of $\partial_1 k$ on every compact set. It is also bounded in $F_s$ since $\partial_1 k$ is $C^{m+1}$ in the second variable. We get the same conclusion as above. 
\\
Since the pointwise derivative $\partial_1 k_qp(v)$ is continuous (by the same argument than for $k_qp$ we have that $\partial_1 k_qp(v)$ is continuous) $d[k_qp] = (\partial_1 k)_qp$.
\\
By recurrence the result is extended to $\partial^n_x kq_p$ for $n \in [1,n+2]$: we obtain that $H = \frac{1}{2}\langle p,k_qp \rangle < \infty$ and $k_qp \in V$.

\vspace{0.2cm}

To prove that the mapping $(p,q) \mapsto k_qp$ is \lip on every compact, the composition is \lip on every bounded ball if $\partial_1k$ is $C^{m+1}$ in the second variable. Hence we deduce that for each $x_0$, the maps $q \in F_s \mapsto k(x_0,q) \in F_s$ and $q \in F_s \mapsto \partial_1 k(x_0,q) \in F_s$ are Lipschitz. The \lip constant can be bounded for $x_0 \in K$ by continuity of the kernel derivatives.
\\
As the dual pairing is Lipschitz we obtain the result. Then by triangular inequality we also obtain that $k_qp \circ q$ is \lip in both variables and so is $\langle p, (\partial_x k_qp) \circ q \rangle$.
\end{proof}

\section{Cylindrical Brownian motion\index{Cylindrical Brownian Motion} and stochastic integral\index{Stochastic integral}} \label{IntSto}
The goal of this section is to define the stochastic integral $\int_0^T u(x) dB_x$ for $u$ a suitable random variable with values in $F_s$. We will give a self-contained presentation inspired by \cite{1140.60034} provided basic knowledge of the It\^{o} stochastic integral. However we aim at presenting it in order to keep up the finite-dimensional approximations for landmarks.
\\
First we provide an elementary and self-contained introduction to the cylindrical Brownian motion in $H$. The construction here puts the emphasis on the finite-dimensional approximations obtained by projection on finite-dimensional subspaces which are the counterpart in the noise model of the finite-dimensional approach with landmarks. We will then present in \ref{stochastic_int} the stochastic integral.

\subsection{Cylindrical Brownian motion in $L^2(S_1,\mathbb{R})$} \label{noise_intro}
We start with the simplest situation where the underlying space in the one dimensional torus $S_1$ and $H=L^2(S_1,\mathbb{R})$.

Let $(B^{n,k})_{n\geq 0,k \in A_n}$ be a collection of continuous independent standard one-dimensional Brownian motions (BM) on $(\Omega,\mathcal{F},P)$ a probability space. For any $n\geq 0$, and consider the $H$ valued random process
$$W^n_t\doteq \sum_{l=-1}^{n-1} \sum_{k \in A_l} B^{l,k}_t\psi_{l,k} \, .$$

At a given time, the coefficients of $W^n_t$ in the orthonormal basis of $H$ are i.i.d. Gaussian variables with variance $t$ (truncated at rank $n$). Moreover, since
\begin{equation}
H^n\doteq \text{Span}\{\, \psi_{l,k}, -1 \leq l< n,\ k \in A_n \, \}\label{eq:26-1}
\end{equation}
is the $2^{n}$-dimensional space of  piecewise constant on the dyadic partition of $S^1$ at scale $2^{-n}$, $W^n_t\in H^n$ is obviously a random piecewise constant function. Moreover, for any $f\in H^n$ $f\doteq \sum_{l=-1}^{n-1}\sum_{k \in A_l} f_{l,k}\psi_{l,k}$, we get 
\begin{equation}
\langle f,W^n_t\rangle_H=\sum_{l=-1}^{n-1}\sum_{k \in A_l} f_{l,k}B^{l,k}_t \sim \mathcal{N}(0,|f|_H^2)\,.\label{eq:15-1}
\end{equation}
More generally, for any $f_1,f_2\in H^n$, $(\langle f_1,W^n_t\rangle_H,\langle f_2,W^n_t\rangle_H)$ are jointly Gaussian, centred with covariance  
\begin{equation}
\Gamma_{i,j}\doteq E(\langle f_i,W^n_t\rangle_H \langle f_j,W^n_t\rangle_H)=\langle f_i,f_j\rangle_H\,.\label{eq:15-2}
\end{equation}
In particular, if we introduce $\phi_{n,k}\doteq 2^{n/2}\mathbf{1}_{[k2^{-n},(k+1)2^{-n}[}$ for any $0\leq k\leq 2^n-1$, the $\phi_{n,k}$'s define an orthonormal basis of $H^n$. Denoting $\gamma^{n,k}_t\doteq \langle \phi_{n,k},W^n_t\rangle_H$, we get from \eqref{eq:15-1} and \eqref{eq:15-2} that 
\begin{equation}
W^n_t=2^{n/2}\sum_{k=0}^{2^n-1}\gamma^{n,k}_t\mathbf{1}_{[k2^{-n},(k+1)2^{-n}[}\label{eq:27-3}
\end{equation}
where $(\gamma^{n,k}_t)_{t\geq 0}$ is a i.i.d. family of $2^n$ standard Brownian motions indexed by $k$.

A cylindrical Brownian motion on $H$ is the limit of $W^n_t$ when $n\to\infty$. A well known but important fact is that this limit is not defined in $H$ since 
$E(|W^{n+j'}_t|^2_n)=t2^n\to +\infty$ but in any $H_{-s}$ for $s>1$. 
\noindent
Indeed, $|W^{n+j}_t-W^{n+j'}_t|^2_{H_{-s}}\leq R_{n,t}^2\doteq \sum_{m=n+1}^\infty 2^{-ms}\sum_{k \in A_m}|B^{m,k}_t|^2$ so that 
$$E(\sup_{j,j'\geq 0}|W^{n+j}_t-W^{n+j'}_t|^2_{H_{-s}})\leq E(R^2_{n,t})=C_{n,s}t$$
with $C_{n,s}\doteq 2^{(n+1)(1-s)}/(1-2^{1-s})$. Therefore $a.s.$, $W^n_t$ is a Cauchy sequence in $H_{-s}$ and one can define $W_t$ as the limit in $H_{-s}$ of $W^n_t$. In fact, it will be helpful to do a little more. Since the process $t\to W^n_t$ has continuous trajectories in $H_s$, one can look for a limit in $C(\mathbb{R}_+,H_s)$ for the uniform topology.

For any $T>0$, we have
$$\sup_{j,j'\geq 0}\sup_{0\leq t\leq T}|W^{n+j}_t-W^{n+j'}_t|_{H_{-s}}^2\leq R_n^2\doteq \sum_{l=n+1}^\infty 2^{-ls}\sum_{k \in A_l}\sup_{0\leq t\leq T}|B^{l,k}_t|^2$$
so that using the Doob inequality $E(\sup_{0\leq t\leq T}B_t^2)\leq 4E(B_t^2)$ for the standard Brownian motion, we get
$$E\left(\sup_{j\geq 0}\sup_{0\leq t\leq T}|W^{n+j}_t-W^{n}_t|_{H_{-s}}^2\right)\leq 4C_{n,s}T\,.$$

Hence a.s. $t\to W^n_t$ is a Cauchy sequence in $C([0,T],H_{-s})$. Since $T>0$ is arbitrary, we can define a limit process $W$ living in $C(\mathbb{R}_+,H_{-s})$ such that for any $T\geq 0$
$$E(\sup_{0\leq t\leq T}|W_t-W^n_t|_{H_{-s}}^2)\leq 4C_{n,s}T\,.$$ 

\subsection{Cylindrical Brownian motion in $L^2(T_m,\mathbb{R})$}
For a general $m\geq 1$, since $L^2(T_m,\mathbb{R})=L^2(T_1,\mathbb{R})\otimes\cdots\otimes L^2(T_1,\mathbb{R})$, the construction of the $W^n$ is built from the
Hilbert basis obtained by usual tensorisation. To be more explicit, we denote by 
$$\psi^m_{l,k}\doteq \otimes_{i=1}^m \psi_{l^{(i)},k^{(i)}}$$
for any $l=(l^{(i)})_{1\leq i\leq m}$ and $k=(k^{(i)})_{1\leq i\leq m}$ such that $l^{(i)}\geq -1$ and $k^{(i)} \in A_{l^(i)}$ for $l^{(i)}\geq -1$. Now, if $I_n\doteq\{\ (l,k)\ |\ l^{(i)}\leq n,\ 1\leq i\leq m\}$ 
$$H^n=\text{span}\{\psi^m_{l,k}\ |\ (l,k)\in I_n\}$$
and $I_\infty=\cup_{n\geq 0}I_n$, we define from a family $(B^{l,k})_{(l,k)\in I_\infty}$ of i.i.d. standard BM 
\begin{equation}
W^n_t\doteq \sum_{(l,k)\in I_n}B^{l,k}_t\psi^m_{l,k} \, .\label{eq:01-2-9}
\end{equation}
As previously, if $f\in H^n$, we have $t\to |f|_H^{-1/2}\langle f,W^n_t\rangle$ is a standard BM and for any $f,g\in H^n$, 
\begin{equation}
(\langle f,W^n_t\rangle \langle g,W^n_t\rangle)=t\langle f,g\rangle_H \,.  \label{eq:01-2-10}
\end{equation}
\\
In particular, if 
$$\phi^m_{n,k}\doteq 2^{mn/2}\mathbf{1}_{\prod_{i=1}^m[k^{(i)}2^{-n},(k^{(i)}+1)2^{-n}[},\ k^{(i)}\in\llbracket 0,2^n\llbracket$$ 
the family $(\phi^m_{n,k})_{k\in{\llbracket 0, 2^n\llbracket}^m}$ is an orthonormal basis of $H^n$ based on a dyadic partition of $T_m$ in cells  of size $2^{-n}\times\cdots \times 2^{-n}$. As previously, we have
\begin{equation}
W^n_t=\sum_{k\in\{0,\cdots 2^n-1\}^m}\langle W^n_t,\phi^m_{n,k}\rangle \phi^m_{n,k}=2^{nm/2}\sum_{k\in\{0,\cdots 2^n-1\}^m}\gamma^{n,k}_t\mathbf{1}_{\prod_{i=1}^m[k^{(i)}2^{-n},(k^{(i)}+1)2^{-n}[}\label{eq:01-2-11}
\end{equation}
where $(\gamma^{n,k})$ is a family of i.i.d. BM.

For $H_s=\otimes_{i=1}^mH_{-s_i}$, with $s=(s_1,\cdots,s_m)$, we get immediately that $W^n_t$ is a Cauchy sequence in $H_{_s}$ as soon as $s_*\doteq \min_{i}s_i>1$ converging uniformly on any time interval $[0,T]$ to a process $W\in C(\mathbb{R}_+,H_{-s})$. More precisely, 
\begin{equation}
  \label{eq:01-2-8}
  E(\sup_{0\leq t\leq T}|W_t-W^n_t|_{H_{-s}})\leq 4C_{n,s}T
\end{equation}
with $C_{n,s_*}=(\sum_{l=n+1}^\infty 2^{-l(s_*-1)})^m$.
\subsection{Cylindrical Brownian motion in $H=L^2(T_m,\mathbb{R}^d)$}
For a general $d\geq 1$, the previous definition on cylindrical Brownian motion can be extended easily in the more general situation where $H=L^2(T_m,\mathbb{R}^d)$. Indeed, we define $W\doteq (W^{(1)},\cdots,W^{(d)})$ where $(W^{(i)})_{1\leq i\leq d}$ is a family of i.i.d. cylindrical Brownian motions in $L^2(T_m,\mathbb{R})$ as defined previously. The finite dimension approximations are defined accordingly on 
\begin{equation}
  \label{eq:01-2-12}
  H^n\doteq\text{span}\{ (\psi^m_{l_1,k_1},\cdots,\psi^m_{l_d,k_d})\ |\ (l_j,k_j)\in I_n\ 1\leq j\leq d\ \}\,.
\end{equation}
In this case, there is an analog of inequality \eqref{eq:01-2-12}, with the constant $$C_{n,s_*}=d(\sum_{l=n+1}^\infty 2^{-l(s_*-1)})^m.$$
\subsection{Stochastic integral}

We assume basic knowledge of the It\^{o} integral and we directly deal with the general case on $H=L^2(T_m, \ms{R}^d)$. We recall that $F_{s} \hookrightarrow H \hookrightarrow F_{-s}$. Having in mind applications that we will develop later, we need to introduce the space of integrands. 
\\
Let us denote by $L(F_{-s})$ the space of continuous endomorphisms of $F_{-s}$. If $u \in L(F_{-s})$, then there exists a constant denoted by $|u|$ such that 
$$ |u(e)|_{F_{-s}} \leq |u||e|_{F_{-s}}.$$
\begin{defi}
The set $L_{T}$ contains all random variables $u: [0,T] \times \Omega \mapsto L(F_{-s})$ verifying,
\begin{itemize}
\item $(t,\omega) \rightarrow u(t,\omega)$ is $\mathcal{B}[0,T] \otimes \mathcal{A}$ measurable, 
\item $ \omega \rightarrow u(t,\omega)$ is $\mathcal F_t-$measurable for $t \in [0,T]$,
\item $\int_0^T E[|u(t)|^2] \, dt < \infty \, .$
\end{itemize} 
\end{defi}

Now, we want to give a sense to
\begin{equation} \label{stochastic_int}
\int_0^T u(t)dW_t \, ,
\end{equation}
for $u \in L_T$.
To this end, we first define
\begin{equation} \label{finite_dim}
\int_0^T u(t)dW_t^n = \sum_{l' \in \ms{N}} (\int_0^T \sum_{l,k}^{l<n} u_{l,k}^{l'} dB_t^{l,k}) e_{l'}  \, ,
\end{equation}
with $(e_{l'})_{l' \in \ms{N}}$ an orthonormal basis of $F_{-s}$. 
Each term $\int_0^T \sum_{l,k}^{l<n} u_{l,k}^{l'} dB_t^{l,k}$ are well-defined since it is a finite sum of It\^{o} integrals and we have with the Doob inequality
\begin{equation}
E[\sup_{t \in [0,T]} |\int_0^t \sum_{l,k}^{l<n} u_{l,k}^{l',k'} dB_t^{l,k}|^2] \leq 4 \sum_{l,k}^{l<n} \int_0^T E[(u_{l,k}^{l'})^2] ds \, .
\end{equation}
Therefore we get,
\begin{multline} \label{Doob_infinite}
E[ \sup_{i' \geq i \geq 0} \sup_{t \in [0,T]} |\int_0^t u(t)dW_t^{n+i'} - \int_0^T u(t)dW_t^{n+i}|^2_{F_{-s}} ] \leq 4 \sum_{l' \in \ms{N}} \sum_{l=n+i}^{n+i'-1} \sum_{k \in A_l} (\int_0^T E[(u_{l,k}^{l'})^2] ds) |e_{l'}|^2_{F_{-s}} \\
\leq (\sum_{l=n+i}^{n+i'-1} \sum_{k \in A_l} |\psi_{n,k}|^2_{F_{-s}}) \int_0^T E[|u(s)|^2] ds \,, 
\end{multline} 
since $$ \sum_{l' \in \ms{N}} (\int_0^T E[(u_{l,k}^{l'})^2] ds) |e_{l'}|^2_{F_{-s}} = \int_0^T E[|u(s)(\psi_{l,k})|^2] ds\,.$$
We also have $|\psi_{n,k}|_{F_{-s}} \leq |\psi_{n,k}|_{H_{-s}}$, and then
$$ \sum_{l=n+i}^{n+i'-1} \sum_{k \in A_l} |\psi_{n,k}|^2_{F_{-s}} \leq \sum_{l=n+i}^{n+i'-1} \sum_{k \in A_l} |\psi_{n,k}|^2_{H_{-s}} \leq (C_{n+i,s_*} - C_{n+i',s_*})  \,.$$
Hence, $\int_0^T u(t)dW_t^n$ is a Cauchy sequence in $C([0,T],F_{-s},|.|_{\infty})$.
\\
The next property is the application of the previous Doob inequality \eqref{Doob_infinite} when $\sigma$ is bounded.
\begin{prop}
Assume that $\sigma \in L_T$ is bounded by $|\sigma|_{\infty}$ then we have,
\begin{align}
E[\sup_{t \in [0,T]} \left( \int_0^t \sigma(s)(dW^{n+l}_s-dW^{n}_s) \right)^2] \leq 4 |\sigma|_{\infty}^2 T (C_{n+i,s_*} - C_{n+i',s_*}) \, .
\end{align}
\end{prop}

\section{Solutions to the SDE on $P \times Q$} \label{generalsolutions}
Recall that $P=F_{-s}$ and $Q=F_s$.
Let $E=P\times Q$ be the phase space equipped with the product Hilbert structure. Considering the injection $i:F_{-s}\to E$ defined by $w\mapsto (w,0)$ and identifying $W$ with $i\circ W$ and $W^n$ with $i\circ W^n$, we can assume that $W$ and the projections $W^n$ are $C(\mathbb{R},E)$-valued. Now, for any $n \geq 0$, we introduce the finite-dimensional subspace $E_n\doteq H_n\times H_n\subset E$ where $H_n$ is given by \eqref{eq:26-1}. We denote also $E_\infty \doteq \cup_{n\geq 0}E_n$ which defines a dense subspace of $E$.
 The space $E_n$ is finite-dimensional and the restriction of the Hamiltonian $H$ on $E_n$ is well defined. Moreover, if the kernel $K(a,b)$ is $C^2$ on each variable, then $H(x)=H(p,q)$ is $C^2$ in the variable $x\in E_n$. We can define the $C^1$ function $f$ on $E_n$ as

\begin{equation}
x\mapsto f(x)\doteq (-\partial_qH(x),\partial_pH(x))^T\in E_n,\quad x\in E_n\,.\label{eq:31-5}
\end{equation}

Let $\sigma:E_\infty\to l_E$ be a Lipschitz map on any ball of $E_\infty$ and let $(X^n)_{0\leq t<\tau}$ be the pathwise continuous solution of 
the SDE 
\begin{equation}
dY_t=f(Y_t)dt+\sigma(Y_t)dW^n_t, \quad Y_0=x^n_0\label{eq:27-2}
\end{equation}
defined until explosion time $\tau^n$.
\noindent
We need to consider the following hypothesis.
\begin{description}
\item[H0] The space $V$ can be continuously embedded in $C^{m+1}_b(\mathbb{R}^d,\mathbb{R}^d)$ ie there exists $C>0$ such that $|v|_{m+1,\infty}\leq C|v|_V$ for any $v\in V$.  
\item[H0'] The trace of the operator induced by on $H_n$ by $\partial^2_{p}H(X^n_s)$ can be controlled as 
\begin{equation} \label{trace_bornée}
\text{tr}(\sigma^T k(Q^n,Q^n) \sigma) \leq c \, .
\end{equation}
\end{description}
Note that if \textbf{H0} holds, then for any $b,b'\in\mathbb{R}^d$, $K(.,b)b'\in V$ and $K$ is $C^{m+1}$ in each variable. Moreover, the second hypothesis \textbf{H0'} will be verified (in lemma \ref{lemme_de_trace}) for $\sigma$ a \lip mapping in $L(F_s)$ with the additional assumption that for every $X \in E$, $\sigma_X(L^1) \subset L^1$ and the norm of this restriction (with the $L^1$ norm) is bounded uniformly in $X$. However it can be interesting to keep this hypothesis for slightly different models.
\begin{prop}\label{prop31-1-2} Under assumption \textbf{H0},
  the explosion time of the SDE \eqref{eq:27-2} is almost surely infinite ie 
$X^n$ is defined for $t\geq 0$ a.s.
\end{prop}
\begin{proof}
  Let $R>0$ be a positive real number and $\tau^n_R=\inf\{\ t\geq 0\ |\ |X^n_t|\geq R\ \}$ (which
is well defined since $X^n$ exists and is continuous until explosion time). We denote by $\tau^n=\lim_{R\to\infty}\tau^n_R$, so that on the event $(\tau^n<\infty)$ 
the solution $X^n_t$ blows up for $t\to\tau^n$.
\noindent
Using Itô
formula for the process $H(X^n_{t\wedge \tau^n_R})$ we get for
$X^n_t=(P^n_t,Q^n_t)$
\begin{multline*}
H(X^n_{t\wedge \tau^n_R})=  H(x^n_0)+\int_0^{t\wedge \tau^n_R}(\partial_qH(X^n_s)dQ^n_s+\partial_p H(X^n_s)dP^n_s) \\
+\frac{1}{2}\langle (\sigma(X^n_s)dW^n)^T (\partial^2_{p}H(X^n_s) \sigma(X^n_s)dW^n \rangle_s \,.
\end{multline*}
Since we have 
\begin{itemize}
\item $\partial_qH(X^n_t)dQ^n_t+\partial_p
H(X^n_t)dP^n_t=\partial H_p(X^n_t)\sigma(X^n_t)dW^n_t$
\item from \textbf{H0'}, we have that almost surely, for all $t$
$$\int_0^t \langle (\sigma(X^n_s)dW^n)^T \partial^2_{p}H(X^n_s) \sigma(X^n_s) dW^n \rangle_s \leq c t \, .$$
\end{itemize}
we get
\begin{equation*}
H(X^n_{t\wedge \tau^n_R})\leq H(x^n_0) + M_{t\wedge \tau^n_R} + \int_0^{t\wedge \tau^n_R}\frac{1}{2} c \, ds 
\end{equation*}
where $M_{t\wedge \tau^n_R}$ is a bounded continuous martingale. 
So that with the hypothesis \textbf{H0'} we have,
$$E(H(X^n_{t\wedge \tau^n_R}))\leq   H(x^n_0)+ \frac{1}{2}c \, (t\wedge \tau^n_R)\,.$$
In particular, 
$E(\int_0^{t\wedge \tau^n_R}H(X^n_s)ds)\leq tH(x^n_0)+\frac{c}{4} \, t^2<\infty$
and using Fatou Lemma
\begin{equation}
E(\int_0^{t\wedge \tau^n}H(X^n_s)ds)\leq tH(x^n_0)+\frac{c}{4}  \, t^2<\infty\label{eq:31-7}
\end{equation}
so that almost surely
\begin{equation}
\int_0^{t\wedge\tau^n_R} H(X^n_s)ds<\infty\,. \label{eq:30-1}
\end{equation}
Now, for $x=(p,q)\in E_n$, we consider 
$$v_x(z)\doteq\int_{T_m}K(z,q(z'))p(z')dz'\in V$$
for which
\begin{equation}
\frac{1}{2}|v_x|_V^2=\frac{1}{2}\int_{T_m\times T_m}p(z)^TK(q(z),q(z'))p(z')dzdz'=H(x)\,.\label{eq:30-2}
\end{equation}
\noindent
From \eqref{eq:30-1} and \eqref{eq:30-2}, we can define by pathwise integration a continuous random process 
$$\Phi^n \doteq (\Phi^n_t)_{0\leq t<\tau_R^n}$$ solution of the flow equation, 

\begin{equation}
\frac{\partial}{\partial t}\Phi^n_t=v_{X^n_t}(\Phi^n_t)\,.\label{eq:31-3}
\end{equation}
\noindent
Assuming a continuous embedding $V\hookrightarrow C^{m+1}(\mathbb{R}^d,\mathbb{R}^d)$, $\Phi^n_{t\wedge \tau^n_R}$ is almost surely a $C^{m+1}$ diffeomorphism and there exists a constant $D$ such that almost surely
\begin{equation}
|\Phi^n_{t\wedge \tau^n_R}|_{m+1,\infty}\leq C\exp(C\sqrt{t}(\int_0^{t\wedge \tau^n_R }H(X^n_s)ds)^{1/2})\,.\label{eq:31-4}
\end{equation}

\emph{In the sequel, we denote by $C$ a generic constant non depending on $n$, $t$ and $R$ possibly changing from line to line.}
\noindent
Thus, since $Q^n_{t\wedge \tau^n_R}=\Phi^n_{t\wedge\tau^n_R}(q^n_0)$ and $\sup_{n\geq 0}|q^n_0|_{F_s}<\infty$,  we get from Theorem \ref{resume_P_Q} and proposition \ref{prop19} that, uniformly in $n$, we have almost surely (for maybe a different but still universal constant $C$, see above)
\begin{equation}
K^n_t\doteq \limsup_{R\to +\infty}|Q^n_{t\wedge\tau^n_R}|_{F_s}\leq C \exp(C\sqrt{t}(\int_0^{t\wedge \tau^n }H(X^n_s)ds)^{1/2})<+\infty\,.\label{eq:31-2}
\end{equation}
In particular, $Q^n_t$ does not blow up for $t\to\tau^n$ on $\tau^n<\infty$. Therefore it is sufficient to show that $P^n_t$ does not blow up as well to get 
by contradiction that $\tau=\infty$ almost surely.

As from the continuous embedding on $V$ in $C^{m+1}$, $|dv_x|_{m,\infty}\leq C|v_x|_V$, we get from proposition \ref{prop19} $|dv_{X_{t\wedge\tau^n_R}}(Q^n_{t\wedge\tau^n})|_{F_s}\leq C|H(X_{t\wedge\tau^n_R})|^{1/2}K^n_t$ and using the continuity of the product in $F_s$ (i.e. there exists $M>0$ such that $|ff'|_{F_s}\leq M|f|_{F_s}|f'|_{F_s}$ for any $f,f'\in F_s$) we obtain  for any $\delta q\in F_s$ 
$$|dv_{X_{t\wedge\tau^n_R}}(Q^n_{t\wedge\tau^n})\delta q|_{F_s}\leq C|H(X_{t\wedge\tau^n_R})|^{1/2}K^n_t|\delta q|_{F_s}\,. $$
Therefore, 

\begin{multline}
\left|\int_{T_m}\langle
  dv_{X_{t\wedge\tau^n_R}}(Q^n_{t\wedge\tau^n}(z))^*
  P^n_{s\wedge\tau^n_R}(z),\delta
  q(z)\rangle_{\mathbb{R}^d}dz\right| \\
 \quad =\left|\int_{T_m}\langle
  P^n_{s\wedge\tau^n_R}(z),dv_{X_{t\wedge\tau^n_R}}(Q^n_{t\wedge\tau^n}(z))\delta
  q(z)\rangle_{\mathbb{R}^d}dz\right|
\quad \leq
C|H(X_{t\wedge\tau^n_R})|^{1/2}K^n_t|P^n_{t\wedge
  \tau^n_R}|_{F_{-s}}|\delta q|_{F_s}
\end{multline}

and
$$|\partial_qH(X^n_{t\wedge \tau^n_R})|_{F_{-s}}\leq  C|H(X_{t\wedge\tau^n_R})|^{1/2}K^n_t|P^n_{t\wedge
      \tau^n_R}|_{F_{-s}}$$
since $\partial_qH(X^n_{t\wedge \tau^n_R})=dv_{X_{t\wedge\tau^n_R}}(Q^n_{t\wedge\tau^n})^*
      P^n_{s\wedge\tau^n_R}$.
We deduce that
$$|P^n_{t\wedge \tau^n_R}|_{F_{-s}}\leq |p^n_0|_{F_{-s}}+CK^n_t\int_0^{t\wedge \tau^n_R}H(X_{s\wedge\tau^n_R})^{1/2}|P^n_{s\wedge \tau^n_R}|_{F_{-s}}+|\int_0^{t\wedge \tau^n_R}\sigma(X^n_s)dW^n_s|_{F_{-s}}$$
and from Gronwall's Lemma
\begin{equation}
|P^n_{t\wedge \tau^n_R}|_{F_{-s}}\leq \left( |p^n_0|_{F_{-s}}+\sup_{u\leq t\wedge\tau^n_R}|\int_0^{u\wedge \tau^n_R}\sigma(X^n_s)dW^n_s|_{F_{-s}}\right)e^{MCK^n_t\sqrt{t}(\int_0^{t\wedge \tau^n_R}H(X_{s\wedge\tau^n_R})ds)^{1/2}}\,.\label{eq:01-2-6}
\end{equation}
Since from the Doob inequality we have for $s_*=\inf_{1\leq i\leq m}s_i$
\begin{equation}
E(\sup_{u\leq t\wedge\tau^n_R}|\int_0^{u\wedge \tau^n_R}\sigma(X^n_s)dW^n_s|_{H_{-s}}^2)\leq 4|\sigma|_\infty^2 t C_{-1,s_*}\label{eq:01-2-7b}
\end{equation}
with a right-hand side independent of $R$, we deduce that almost surely
\begin{equation}
\sup_{u\leq t\wedge\tau}|\int_0^{u\wedge \tau^n_R}\sigma(X^n_u)dW^n_u|_{F_{-s}}\leq \sup_{u\leq t\wedge\tau}|\int_0^{u\wedge \tau^n_R}\sigma(X^n_u)dW^n_u|_{H_{-s}}<+\infty\,.\label{eq:31-1}
\end{equation}
and from \eqref{eq:30-1}, \eqref{eq:31-2} and \eqref{eq:31-1}, we get almost surely
$$\sup_{R\to\infty}|P^n_{t\wedge \tau^n_R}|_{F_{-s}}<+\infty\text{ and }\tau^n > t\,.$$
 \end{proof}

\begin{prop}\label{prop31-1-1} Let $f$ be defined by \eqref{eq:31-5} and assume that \textbf{H0-H'1-H2} hold. Then for any $n\geq 0$ there exists a unique strong solution $X^n=(P^n,Q^n):\Omega\to C(\mathbb{R}_+,E)$ to
\begin{equation}
X^n_t=X^n_0+\int_0^t f(X^n_s)ds +\int_0^t\sigma(X^n_s)dW^n_s,\ X^n_0\equiv x^n_0\in E_n\label{eq:31-1-1a}
\end{equation}
and a random solution  $X=(P,Q):\Omega\to C(\mathbb{R}_+,E)$ to
\begin{equation}
X_t=X_0+\int_0^t f(X_s)ds +\int_0^t\sigma(X_s)dW_s,\ X_0\equiv x_0\in E\label{eq:31-6}
\end{equation}
such that almost surely:
$$\sup_{0\leq t\leq T}|X_t-X^n_t|\to 0\,.$$
 \end{prop}
 \begin{proof}
   From proposition \ref{prop31-1-1}, we know the existence of the finite-dimensional approximation solution $(X^n_t)_{t\geq 0}$ for $t\geq 0$. Moreover, we know from proposition \ref{prop:26-2} the existence of maximal solution $(X_t)_{0\leq t<\tau}$ of the SDE \eqref{eq:31-6} $(X_t)_{0\leq t<\tau}$ up to a possibly finite explosion stopping time $\tau$. 
Moreover,  for any $T>0$ and any $r>0$ we have almost surely
\begin{equation}
\sup_{t\leq T}|X^n_{t\wedge \tau_r}-X_{t\wedge \tau_r}|\to 0\quad \text{ and }
 \sup_{0\leq t\leq T}|\int_0^{t\wedge\tau_r}\sigma(X^n_s)dW^n_s-\int_0^{t\wedge\tau_r}\sigma(X_s)dW_s|\to
0 \label{eq:01-2-1}
\end{equation}
where $\tau_r=\inf\{t\geq 0\ |\  |X_t|\geq r\}$. What we need to prove is that there is no explosion ie $P(\tau<+\infty)=0$ or equivalently, $\tau_r\to +\infty$ almost surely.
\noindent
We start from inequality \eqref{eq:31-7} in the proof of proposition \eqref{prop31-1-2}. Using the uniform convergence \eqref{eq:01-2-1} and Fatou's lemma, we get
\begin{equation}
  \label{eq:01-2-2}
  E(\int_0^{t\wedge \tau}H(X_s)ds)\leq tH(x_0)+\frac{c}{4} t^2<\infty \, .
\end{equation}
Similarly starting from \eqref{eq:31-2}, we get that
\begin{equation}
  \label{eq:01-2-3}
  K_t\doteq \limsup_{R\to +\infty}|Q_{t\wedge\tau_R}|_{F_s}\leq C \exp(C\sqrt{t}(\int_0^{t\wedge \tau }H(X_s)ds)^{1/2})<+\infty\,.
\end{equation}
Moreover, from \eqref{eq:01-2-6}, we get for $n\to \infty$,
\begin{equation}
|P_{t\wedge \tau_r}|_{F_{-s}}\leq \left( |p_0|_{F_{-s}}+\sup_{u\leq t\wedge\tau_r}|\int_0^{u\wedge \tau_r}\sigma(X_s)dW_s|_{F_{-s}}\right)e^{MCK_t\sqrt{t}(\int_0^{t\wedge \tau_r}H(X_{s\wedge\tau_r})ds)^{1/2}}\,.\label{eq:01-2-7}
\end{equation}
Since as in \eqref{eq:01-2-7b}  Doob inequality gives for $s_*=\inf_{1\leq i\leq m}s_i$
$$E(\sup_{u\leq t\wedge\tau_r}|\int_0^{u\wedge \tau_r}\sigma(X_s)dW_s|_{F_{-s}}^2)\leq 4|\sigma|^2_{\infty} tC_{-1,s_*}$$
there exists a random constant $A_t>0$ independent of $r$ such that almost surely
$$|P_{t\wedge \tau_r}|_{F_{-s}}\leq A_t\,.$$
In particular $(\tau\leq t)\subset (\tau_r\leq t)\subset (\max\{K_t,A_t\}\geq r)$ and considering the limit $r\to\infty$, we get $P(\tau\leq t)=0$.
\end{proof}

\subsection{A trace Lemma\index{trace Lemma}}
We now present a sufficient condition to fulfill the hypothesis \textbf{H0'}. With additional assumptions on the kernel and on the variance term, we give a bound for the bracket of the stochastic term of the SDE on finite-dimensional subspaces $H_n$. 
\begin{lm} \label{lemme_de_trace}
Let $k$ be a kernel bounded on the diagonal i.e. there exists $c>0$ such that $b^Tk(a,a)b\leq c|b|^2$ for any $a,b\in \mathbb{R}^d$ or equivalently $k(a,a)\leq c\text{Id}_d$ as a symmetric non-negative bilinear form on $\mathbb{R}^d$.
We assume also that $\sigma(L^1) \subset L^1$ and this restriction is continuous i.e. there exists $M>0$ such that $|\sigma(f)|_{L_1} \leq M |f|_{L^1}$.
Then we have, with $H$ the usual Hamiltonian
$$\int_0^t \langle (\sigma dW_s^n)^T \partial^2_{p} H \sigma dW_s^n \rangle_s \leq c^2M^2 \, t\, .$$
\end{lm}

\begin{proof}
We consider the orthonormal basis $(\phi_{n,k})_{k \in [0,2^{n-1}]^r}$ with $A_n=[0,2^{n-1}]^r$ to write the Hamiltonian as:
$$H = \frac{1}{2^{2nr+1}} \sum_{(i,j) \in A_n^2} p_i^Tk(q_i,q_j)p_j \, ,$$ with $q = \sum_{i \in A_n} q_i 2^{-nr/2}\phi_{n,i} $ and $p = \sum_{i \in A_n} p_i 2^{-nr/2}\phi_{n,i} $.
In this basis, the $L^2$ scalar product can be written as $\langle p , q \rangle = 2^{-nr} \sum_{i \in A_n} p_i q_i$. 
We can write $\sigma dW^n = \sum_{i \in A_n} (\sum_{j \in A_n} \alpha_{i,j} dW^j) 2^{nr/2}\phi_{n,i}$ with $\alpha_{i,j} \in L(\ms{R}^d)$ and $(W^j)_{j \in A_n}$ i.i.d. standard BM with values in $\ms{R}^d$.
Then we have,
$$ \langle (\sigma dW_s^n)^T \partial^2_{p} H \sigma dW_s^n \rangle_s = \frac{1}{2^{2nr}}[\sum_{i,j \in A_n^2} \sum_{h \in A_n}(\alpha_{i,h}^T k(q_i,q_j)\alpha_{i,h})] ds \, .$$
Thanks to the hypothesis on the kernel, we have for any $x,y \in \ms{R}^d$ that $$|a^Tk(x,y)b| \leq \sqrt{a^Tk(x,x)a}\sqrt{b^Tk(y,y)b} \leq c^2 |a||b|$$ and then,
$$|\alpha_{i,h}^T k(q_i,q_j)\alpha_{i,h}| \leq c^2 |\alpha_{i,h}| |\alpha_{j,h}\,| .$$
Now, we can write with Cauchy Schwarz inequality
\begin{multline}\label{maj_accroiss} 
d\langle (\sigma dW_s^n)^T \partial^2_{p} H \sigma dW_s^n \rangle_s \leq c^2 \frac{1}{2^{2nr}} \sum_{h\in A_n} (\sum_{i\in A_n} |\alpha_{i,h}|)^2 ds \leq \sum_{h \in A_n} c^2 |\sigma (\phi_{n,h})|_{L^1}^2 \, ds \, , \\
\leq c^2 M^2 (\sum_{h\in A_n} |\phi_{n,h}|_{L^1}^2) \, ds \leq c^2M^2 \,ds \, ,
\end{multline}
since for any $h\in A_n$, $|\phi_{n,h}|_{L^1}^2=2^{-nr}$.
Note that the inequality \ref{maj_accroiss} is a little abusive but it is to be understood as an inequality on measures with density w.r.t. the Lebesgue measure.
\end{proof}
Remark that we do not need to assume that $\sigma: P \times Q \mapsto L_T$ is bounded by $|\sigma|_{\infty}$, $\sigma(H_n) \subset H_n$. This hypothesis is only required for the existence and uniqueness in all time but not to bound the trace of the operator.
\\
The assumption on the kernel is not restrictive in our range of applications with kernels such as Gaussian kernel or Cauchy kernel. However, the assumption on $\sigma$ is much more demanding. However a wide range of linear maps can be reached. For instance, the convolution with a smooth function is a continuous operator on $F_s$ then by duality it gives a continuous operator on $F_{-s}$. This operator has a continuous restriction to $L^1$.
\\
An important point is that this Lemma covers the case where $\sigma$ is the multiplication by an element of $F_s$.

\section{Approximation Lemmas} \label{Approx_lemmas}
Let $f:E_\infty \to E$ be a function on $E_\infty$ such that $f(E_n)\subset E_n$ for any $n\geq 0$. 
Let also $\sigma : P \times Q \mapsto L(P)$ be a \lip function. Assume that for any $n\geq 0$, we have a random variable  $X^n:\Omega\to C(\mathbb{R}_+,E)$ solution of the stochastic integral equation
\begin{equation}
X^n_t=X^n_0+\int_0^t f(X^n_s)ds +\int_0^t\sigma(X^n_s)dW^n_s,\ X^n_0\equiv x^n_0\in E_n\,.\label{eq:18-1-1}
\end{equation}

\begin{description}
\item[H1] The functions $f$ and $\sigma$ are Lipschitz on $E_\infty$ and can be uniquely extended as Lipschitz functions on $E$. Moreover $\sigma$ is bounded.
\item[H2]  For some $\alpha>1$, we have $\sum_{n\geq 0}n^{2\alpha}|x_0^{n+1}-x^n_0|^2<\infty$.
\end{description}
\begin{prop}
\label{prop:26-1}
  Let $s>1$ be a real number. Under hypothesis $\mathbf{(H1-H2)}$, there exists a random solution 
$X:\Omega\to  C(\mathbb{R}_+,E)$ to
$$X_t=X_0+\int_0^t f(X_s)ds +\int_0^t\sigma(X_s)dW_s,\ X_0\equiv x_0\in E\,$$
such that for any $T>0$, we have almost surely:
  \begin{equation}\label{eq:01-2-5}
\left\{
  \begin{array}[h]{l}
    \sup_{0\leq t\leq T}|X^n_t-X_t|\to
0\,,\\\\
 \sup_{0\leq t\leq T}|\int_0^{t}\sigma(X^n_s)dW^n_s-\int_0^{t}\sigma(X_s)dW_s|\to
0\,.
  \end{array}
\right.
\end{equation}
\end{prop}
\begin{proof} Let $n\geq 0$ be a positive integer and $K$ be an upper bound of the Lipschitz constant for $f$ and $\sigma$. We have 
  \begin{equation}
|X^{n+1}_t -X^{n}_t|\leq |x^{n+1}_0-x^{n}_0|+K\int_0^t|X^{n+1}_s -X^{n}_s|ds +|M^{n+1}_t-M^n_t|\,\label{eq:26-5}
\end{equation}
with $M^n_t=\int_0^t\sigma(X^n_s)dW^n_s$.
Let us consider the last right-hand martingale term. We have 
$$M^{n+1}_t-M^n_t=\int_0^t(\sigma(X^{n+1}_s)-\sigma(X^n_s))dW^{n+1}_s+\int_0^t\sigma(X^n_s)d(W^{n+1}-W^n)_s\,.$$
Using the Doob inequality, we get
\begin{multline}
    E(\sup_{u\leq
      t}|\int_0^u(\sigma(X^{n+1}_s)-\sigma(X^n_s))dW^{n+1}_s|^2_{F_{-s}})\leq
    E(\sup_{u\leq
      t}|\int_0^u(\sigma(X^{n+1}_s)-\sigma(X^n_s))dW^{n+1}_s|^2_{H_{-s}})\\
\leq 4C_{-1,s_*}K^2E(\int_0^u|X^{n+1}_s-X^n_s|^2ds)\label{eq:26-2}
\end{multline}
and with $\delta_{n,s_*} =  C_{n,s_*} -  C_{n+1,s_*}$
\begin{equation}
  \begin{split}
    E(\sup_{u\leq t}|\int_0^u\sigma(X^n_s)d(W^{n+1}-W^n)_s|^2_{F_{-s}})&\leq
    E(\sup_{u\leq t}|\int_0^u\sigma(X^n_s)d(W^{n+1}-W^n)_s|^2_{H_{-s}})\\
&\leq 4|\sigma|^2_{\infty}\delta_{n,s_*} \,.\label{eq:26-3}
  \end{split}
\end{equation}
Thus, for $Z^{n}_t\doteq \sup_{u\leq t}|X^{n+1}_u-X^n_u|^2$, we have

\begin{equation}
E(\sup_{u\leq t} |M^{n+1}_u-M^n_u|^2)\leq 8(K^2C_s\int_0^u E(Z^{n}_s)ds+|\sigma|^2_{\infty} \delta_{n,s_*})\label{eq:26-6}
\end{equation}
and using the inequality $(a+b+c)^2\leq 3(a^2+b^2+c^2)$ and \eqref{eq:26-5}, we get
\begin{equation}
E(Z^{n}_t)\leq 3(|x_0^{n+1}-x_0^n|^2+ 8|\sigma|^2_{\infty}\delta_{n,s_*} +K^2(8C_{-1,s_*}+1)\int_0^tE(Z^{n}_s)ds)\,.\label{eq:26-7}
\end{equation}

Applying Gronwall's Lemma, we get for a sufficiently large constant $A$
\begin{equation}
E(Z^{n}_t)\leq A(|x_0^{n+1}-x_0^n|^2+|\sigma|^2_{\infty}\delta_{n,s_*})\exp(At)\label{eq:26-8}
\end{equation}
and from \textbf{H2}
$$\sum_{n\geq 0}P(\sup_{s\leq t}|X^{n+1}_s-X^n_s|\geq n^{-\alpha})\leq \sum_{n\geq 0}n^{2\alpha}E(Z^n_t)<\infty\,.$$
Borel-Cantelli Lemma gives a.s. $\sup_{s\leq t}|X^{n+1}_s-X^n_s|<n^{-\alpha}$ for $n$ large enough so that $X^n$ converges uniformly on any compact interval $[0,t]$ to a $C(\mathbb{R}_+,E)$-valued process $X$. Similarly from \eqref{eq:26-6}, \eqref{eq:26-8} and \textbf{H2} we get 
$$\sum_{n\geq 0}P(\sup_{s\leq t}|M^{n+1}_s-M^n_s|\geq n^{-\alpha})\leq \sum_{n\geq 0}n^{2\alpha}E(\sup_{s\leq t}|M^{n+1}_s-M^n_s|^2)<\infty$$
and $M^n$ converges uniformly on any compact interval $[0,t]$ to a limit $C(\mathbb{R}_+,E)$-valued process $M$ for which
$$X_t=x_0+\int_0^tf(X_s)ds+M_t\,.$$ 
Let us check that
$$M_t=\int_0^t \sigma(X_s)dW_s\,.$$ 
Indeed, since 
$$E(\sup_{u\leq t}|\int_0^u \sigma(X_s)dW^s-M^n_u|^2)\leq 8(E(\int_0^t K^2|X_s-X^n_s|^2ds)+|\sigma|^2_{\infty}C_{n,s_*}$$
we get for $n\to 0$, $E(\sup_{u\leq t}|\int_0^u \sigma(X_s)dW^s-M_u^n|^2)=0$.
\end{proof}
We extend now the previous result to locally Lipschitz drift $f$ and diffusion $\sigma$.
\begin{description}
\item[H1'] The functions $f$ and $\sigma$ are Lipschitz on any ball of $E_\infty$ and can be uniquely extended as locally Lipschitz functions on $E$. 
\end{description}
\begin{prop}
\label{prop:26-2}
  Let $s>1$ be a positive real number. Under the hypothesis $\mathbf{(H1'-H2)}$, there exists a stopping time $\tau$ and a continuous adapted process $(X_t)_{0\leq t<\tau}$ with values in $E$ such that
\begin{enumerate}
\item $\limsup_{t\to\tau^-}|X_t|=+\infty$ on $(\tau<\infty)$ ($\tau$ is the explosion time)~;
\item for any $r>0$ and any $T>0$, we have almost surely

  \begin{equation}\label{eq:01-2-4}
\left\{
  \begin{array}[h]{l}
    \sup_{0\leq t\leq T}|X^n_{t\wedge\tau_r}-X_{t\wedge\tau_r}|\to
0\\\\
 \sup_{0\leq t\leq T}|\int_0^{t\wedge\tau_r}\sigma(X^n_s)dW^n_s-\int_0^{t\wedge\tau_r}\sigma(X_s)dW_s|\to
0
  \end{array}
\right.
\end{equation}
where $\tau_r\doteq \inf\{t\geq 0\ |\ |X_t|\geq r\}$.
\end{enumerate}
Moreover, for any $t\geq 0$
\begin{equation}
X_{t\wedge\tau_r}=x_0+\int_0^{t\wedge \tau_r} f(X_s)ds +\int_0^{t\wedge\tau_r}\sigma(X_s)dW_s\quad a.s.\label{eq:27-1}
\end{equation}

\end{prop}
\begin{proof}
Let $R>r>0$ be two positive real numbers and 
$$g(x)\doteq \max(\min(1, R-|x|),0)$$ be
a Lipschitz function such that $g^R(x)=1$ if $|x|\leq R$ and $g^R(x)=0$ if $|x|\geq R+1$. We introduce also $f^R=g^Rf$ and $\sigma^R=g^R\sigma$. From \textbf{H1'} $f^R$ and $\sigma^R$ are Lipschitz and we
get from standard results on existence of finite-dimensional SDE a
solution $X^{R,n}\in C(\mathbb{R}_+,E_n)$ of
$dY_t=f^R(Y_s)ds+\sigma^R(Y_s)dW^n_s$ and $Y_0=x^n_0$. From
Proposition \ref{prop:26-1} applied to $f^R$ and $\sigma^R$, there
exists a $C(\mathbb{R}_+,E)$-valued process $X^{R}$ solution of
$dY_t=f^R(Y_s)ds+\sigma^R(Y_s)dW_s$ and $Y_0=x_0$ such that
$P(\sup_{s\leq t}|X^{R}_s-X^{R,n}_s|\to 0)=1$ and 
$P(\sup_{s\leq t}|\int_0^s\sigma(X^R_u)dW_u-\int_0^s\sigma(X^{R,n}_u)dW^n_u|\to 0)=1$. Since $f^R$ and $f$
(resp. $\sigma^{R}$ and $\sigma$) coincide on $|x|\leq R$, we get for
any $n\geq 0$ that
$$X^{R,n}_{t\wedge \tau^{R,n}_R}=X^{n}_{t\wedge \tau^{R,n}_R}\ a.s.$$
  where $\tau^{R,n}_R=\inf\{t\geq 0\ |\ |X^{R,n}_t|\geq R\}$. In
particular $\tau^{R,n}_R=\tau^n_R\doteq \inf\{t\geq 0\ |\
|X^{n}_t|\geq R\}$ almost surely and for any $T\geq 0$,
$P(\sup_{0\leq t\leq T}|X^n_{t\wedge\tau^R_r}-X^R_{t\wedge\tau^R_r}|\to
0)=1$ with $\tau^R_r\doteq \inf\{t\geq 0\ |\ |X^R_t|\geq R\}$ since the uniform convergence of $X^{R,n}$ to $X^R$ on compact set implies that a.s. $\tau^n_R>\tau^R_r$ for $n$ large enough. As a consequence, for two solutions $X^R$ and $X^{R'}$ for $R'>r$, we have $\tau^R_r=\tau^{R'}_r$ a.s. and the trajectories before the common value $\tau^R_r$ are equal. 
Let $R_k$ be an increasing sequence converging to $+\infty$ and $\tau=\lim_{k\to\infty}\tau^{R_k}_{R_k}$. If 
$X_t=\sum_{k=0}^\infty X^{R_{k+1}}_t\mathbf{1}_{\tau^{R_{k}}_{R_k}\leq  t<\tau^{R_{k+1}}_{R_{k+1}}}$ for $t\leq \tau$, the process $(X_t)_{0\leq t<\tau}$ verifies (1), (2) and \eqref{eq:27-1}. 
\end{proof}

\section{Applications and numerical simulations} \label{applications}
This section will present a direct application of the SDE we have studied. In this simplest model, we suppose a shape to be given with an initial momentum and we model the perturbation term with a white noise on the initial shape. Therefore the variance of the noise term is constant in time. 
\\
Let $(q_0,p_0)$ be respectively the initial shape and the initial momentum of the system. As in the deterministic matching with a sufficiently smoothing attachment term, the momentum is a normal $L^2$ vector field on the shape, it is relevant enough for applications to consider $q_0 \in F_s$ and $p_0 \in F_{-s}$. To assume $q_0 \in F_{s}$ means that we chose a parameterization of the shape by $T_m$. We would like our stochastic system to be independent of this initial parameterization. Then we need to understand what is the reparameterization transformation on the deterministic system and on the white noise.
\\
Assume that $\phi$ is a diffeomorphism of $T_m$, then we give the correspondence between the solution $(p_t,q_t)$ from initial conditions $(p_0,q_0)$ and $(\tilde p_t, \tilde q_t)$ with the initial position variable $q_0 \circ \phi$.
$$
\xymatrix{
    T_m \ar[r]^{\phi} \ar[rd]_{q_0 \circ \phi} & T_m \ar[d]^{q_0} \\
     & \ms{R}^d
  }
$$
As the trajectory is entirely determined by the vector field $v_{p_t,q_t}$, the change of variable by $\phi$ gives the correspondence:
\begin{align*}
& \tilde q_t = q_t \circ \phi \, ,\\
& \tilde p_t = \Jac(\phi) \, p_t \circ \phi \, .
\end{align*}
We will denote by $\phi^*(p)$ the pull-back of $p$ under $\phi$. 
\\
The stochastic system verifies the same transformation and the pull-back of the cylindrical Brownian motion is given by
$$ \tilde B \doteq \Jac(\phi) \, B \circ \phi \, .$$
We need the following proposition,
\begin{prop} \label{transfoBrownian}
If $B$ is a cylindrical Brownian motion on $L^2(T_n,\nu)$ with $\nu \ll \mu$ ($\mu$ is the Lebesgue measure) then $\phi^*(B_t) = \Jac(\phi) B_t(\phi(s))$ is a cylindrical Brownian motion on $L^2(T_n,\nu')$ with $\frac{d\nu'}{d\nu} = \frac{1}{\Jac(\phi)}$.
\\
Moreover if $f \in L^2$ with $f \neq 0$ a.e., $f\,B$ is a cylindrical Brownian motion for the measure $\frac{1}{f^2} d\nu$.
\end{prop}

\begin{proof}
If $(e_i)_{i \in \ms{N}}$ is an orthonormal basis of $L^2(T_n,\nu)$, then $\sqrt{\Jac(\phi)} e_i \circ \phi$ is also an orthonormal basis for $L^2(T_n,\nu)$ by a change of variable with $\phi$. As a result, $\Jac(\phi) e_i \circ \phi$ is an orthonormal basis for $L^2(T_n,\nu')$ with $\frac{d\nu'}{d\nu} = \frac{1}{\Jac(\phi)}$.
\\
The second part of the proposition is also straightforward: if $(e_i)_{i \in \ms{N}}$ is an orthonormal basis for $L^2(T_m,\ms{R}^d)$, then $(fe_i)_{i \in \ms{N}}$ is an orthonormal basis for $L^2(T_m,\frac{1}{f^2}d\nu)$. 
\end{proof}

\begin{corr} \label{corr1}
The random process $\phi^*(B_t)$ is equal to $\sqrt{\Jac{(\phi)}}W_t$ with $W_t$ a cylindrical Brownian motion on $L^2(T_m,\ms{R}^d)$.
\end{corr}

Back to our framework with $F_s \times F_{-s}$, we remark that the space $F_s$ is not invariant under a change of coordinates: let $\phi$ be a diffeomorphism of $T_n$ then a priori, if $q \in F_s$ then $q \circ \phi$ may not belong to $F_s$. However if $q$ and $\phi$ are sufficiently smooth then $T(q):=q \circ \phi$ belongs to $F_s$. Hence there exist large subspaces in $F_s$ invariant under this transformation. To go further we could prove that if $s<s'$ then $T(F_{s'}) \subset F_s$.
We would have the same result for the dual spaces: $T(F_{-s}) \subset F_{-s'}$. Furthermore we would like to deal with piecewise diffeomorphisms, this is why the formulation of the following proposition is a little more general.
\begin{prop} \label{reparameter}
Let $q \in F_s(T_n,\ms{R}^k)$ and $\phi : T_n \mapsto T_n$ be a measurable invertible mapping (i.e. there exists $\phi^{-1} :T_n \mapsto T_n$ measurable such that $\phi \circ \phi^{-1} = \text{Id} \; a.e.$). We assume also that $q \circ \phi \in F_s$, $p \circ \phi \in F_{-s}$ and $J:=\frac{d\phi_*^{-1}\mu}{d\mu} \in F_s(T_n,\ms{R})$ such that $J>\ve>0$ a.e. Finally, let $B$ be a cylindrical Brownian motion. 
If $(p_t,q_t)$ is the solution of the system 

\begin{equation} \label{system_simple}
\begin{cases}
dp_t = -\partial_q H(p_t,q_t) + \ve dB_t \\
dq_t = \partial_p H(p_t,q_t) \, ,
\end{cases}
\end{equation}
(with $\ve$ a constant parameter) for initial data $(p_0,q_0)$ and for the path of the white noise $B$ then $(\phi^*(p_t),q_t \circ \phi)$ is the solution of the system 
\begin{equation} \label{system_changed_0}
\begin{cases}
dp_t = -\partial_q H(p_t,q_t) + \ve \sqrt{\Jac(\phi)} d\tilde B_t \\
dq_t = \partial_p H(p_t,q_t) \, ,
\end{cases}
\end{equation}
for initial data $(J\phi^*(p_0),q_0\circ \phi)$ and for the random process $\sqrt{J} \tilde B$, with $\tilde B$ a cylindrical Brownian motion.
\end{prop}

The random process $\sqrt{J} \tilde B$ can be treated in our framework with the map $\sigma: E_n \mapsto L(F_s)$ given by the multiplication with $p_n(\sqrt{J})$: As $J > \ve$ we have that $\sqrt{J} \in F_s$ by smoothness of the square root outside $0$. Thus $\sigma$ is a \lip map such that $\sigma(H_n) \subset H_n$ since the projection is \lip and the multiplication in $F_{-s}$ by an element of $F_s$ is \lip too. It leads to

\begin{theo} \label{theo_first_model}
Under assumption \textbf{H0}, let $f \in F_s$ and $(p_0,q_0) \in F_{-s} \times F_s$ be initial conditions verifying that there exists $s''<s$ and $s'>s$ such that 
$q_0 \in F_{s'}$ and $p_0 \in F_{s''}$.
Then random solutions of the following system with initial conditions $(p_0,q_0)$ 
\begin{equation} \label{system_changed_1}
\begin{cases}
dp_t = -\partial_q H(p_t,q_t) + f d B_t \\
dq_t = \partial_p H(p_t,q_t) \, ,
\end{cases}
\end{equation}
are defined for all time and there is an almost sure convergence of the approximations (also defined for all time)
\begin{equation} \label{system_changed_2}
\begin{cases}
dp_t^n = -\partial_q H(p_t^n,q_t^n) + f^n d B_t \\
dq_t^n = \partial_p H(p_t^n,q_t^n) \, ,
\end{cases}
\end{equation}
 to the previous random solution with initial conditions $(p_0^n,q_0^n)$ the projection on $E_n$ of $(p_0,q_0)$.
\end{theo}

\begin{proof}
This is the application of proposition \ref{prop31-1-1}. We verify hypothesis \textbf{H0'} and control the trace of the operator. 
Remark that we need to control the Hilbert-Schmidt norm of the operator $\sigma^T k(q,q) \sigma$ with $\sigma$ the multiplication by an element of $F_s$. 
This is a consequence of Lemma \ref{lemme_de_trace}, but we give here a simpler proof, since the multiplication is a diagonal operator.
We have
\begin{equation*}
|f^n|_{\infty} \leq C_s|f^n|_{F_s} \leq C_s |f|_{F_s} \, ,  
\end{equation*}
then we get for any $q \in F_s$, 
$$0 \leq \int_{T_m} f^n(s)^T k(q(s),q(s)) f^n(s) ds \leq  C_s^2 |k|_{\infty} |f|_{F_s}^2 \text{Vol}(T_m) \, .$$
Thus hypothesis \textbf{H0'} is verified.
\\
From the assumption on the initial conditions, if $q_0 \in F_{s'}$ with $s' > s$ we have
$$ |q_0|_{F_{s'}}^2 = \sum_{n=0}^\infty 2^{n(s'-s)} |q_0^{n+1} - q_0^n|_{F_s}^2\, < +\infty \, .$$
Moreover, if $p_0 \in F_{s''}$ then 
$$ |p_0|_{F_{s'}}^2 = \sum_{n=0}^\infty 2^{n(s-s'')} |p_0^{n+1} - p_0^n|_{F_s}^2\, < +\infty \, .$$
Hence \textbf{H2} is verified.
\end{proof}
\noindent
Remark that \textbf{H2} is not so demanding as proved in this theorem.

\vspace{0.3cm}
\noindent
Now we can discuss some basic situations to simulate the stochastic system \eqref{system_simple}. The preceding result will enable us to deal with a wide family of shapes and noises. We first develop the case of curves.

\begin{prop}
Let $1<s<2$ a real number and $f: S_1 \mapsto \ms{R}^2$ be a piecewise $C^2$ mapping such that $|f'| > \ve >0$ with $\ve$ a real positive number. If we denote by $\alpha$ the arc-length parameterization of $f$ then there exists a piecewise affine homeomorphism $\phi : S_1 \mapsto S_1$ such that $\phi \in F_s$ and $\alpha \circ \phi$ is in $F_s$. 
\end{prop}

\begin{proof}
Let us assume that the arc-length parameterization $\alpha$ has $p$ singularities at points $x_1 < \ldots < x_n \in S_1$. Then, we define $\phi$ as the linear interpolation on $n$ dyadic points in $S_1$,
$d_1 < \ldots < d_n$ with for images respectively $x_1 < \ldots < x_n$. As $\phi$ is piecewise affine it belongs to $F_s$ thanks to proposition \ref{lipstability}. With the same argument we conclude that $\alpha \circ \phi \in F_s$. 
\end{proof}

The previous proposition tells us that we can consider our stochastic system on an initial shape with a white noise which is white with respect to the arc length.
Hereunder are some simulations to illustrate the convergence of the landmark simulation and the kind of trajectories generated by this model. The figure Fig.~\ref{simu1} presents the convergence of the image of the unit circle under the flow generated by the system with an increasing number of landmarks. We chose to illustrate this convergence since in some sense it just shows the convergence of the vector fields generated by the stochastic system.

\begin{figure}[htbp]
	\centering
		\includegraphics[width=8cm]{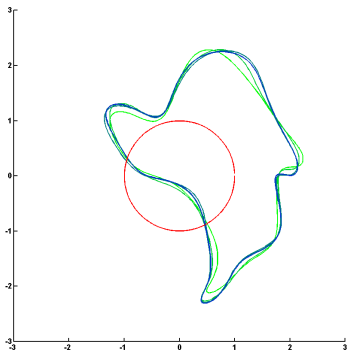}
	\caption{\footnotesize{The convergence when increasing the number of landmarks. From $32$ landmarks to $512$ landmarks for $q_0$ the unit circle and $p_0=0$, the curves from green to blue show the image of the unit circle under the flow for the $2^i$ particles for $i\in [5,9]$. The width of the Gaussian kernel is $0.5$ and the intensity of the noise is $\ve=1.5$. 
}
}
\label{simu1}
\end{figure}

We may also want to see how the shapes are distributed around the target shape, or to learn something about a neighborhood of a shape in this stochastic model. 
We plotted few simulations of the model for the unit circle as initial shape and an initial momentum which is null to see how the neighborhood of the circle looks like. In the figure Fig.~\ref{vois} the red curve is the unit circle and the other curves from green to blue are random deformations of the unit circle for $200$ landmarks.
The kernel size is $0.7$ and the standard deviation of the noise (normalised with the number of particles) is relatively high at $\ve=3$. 
\\
The last simulation shows again the convergence of the landmark discretization as in figure Fig.~\ref{simu1} but with a structure of noise which is null on the particles $q_i$ initially such that $p_x(q_i) <0$ (i.e. with negative abscissa). It illustrates the locality of the noise and we can remark the size (variance) of the kernel in this simulation. In this case the kernel size is smaller at $0.25$ and the standard deviation of the noise is $1.5$.

\begin{figure}[htbp]
 \begin{minipage}{.45\linewidth}
  \centering	\includegraphics[width=7cm]{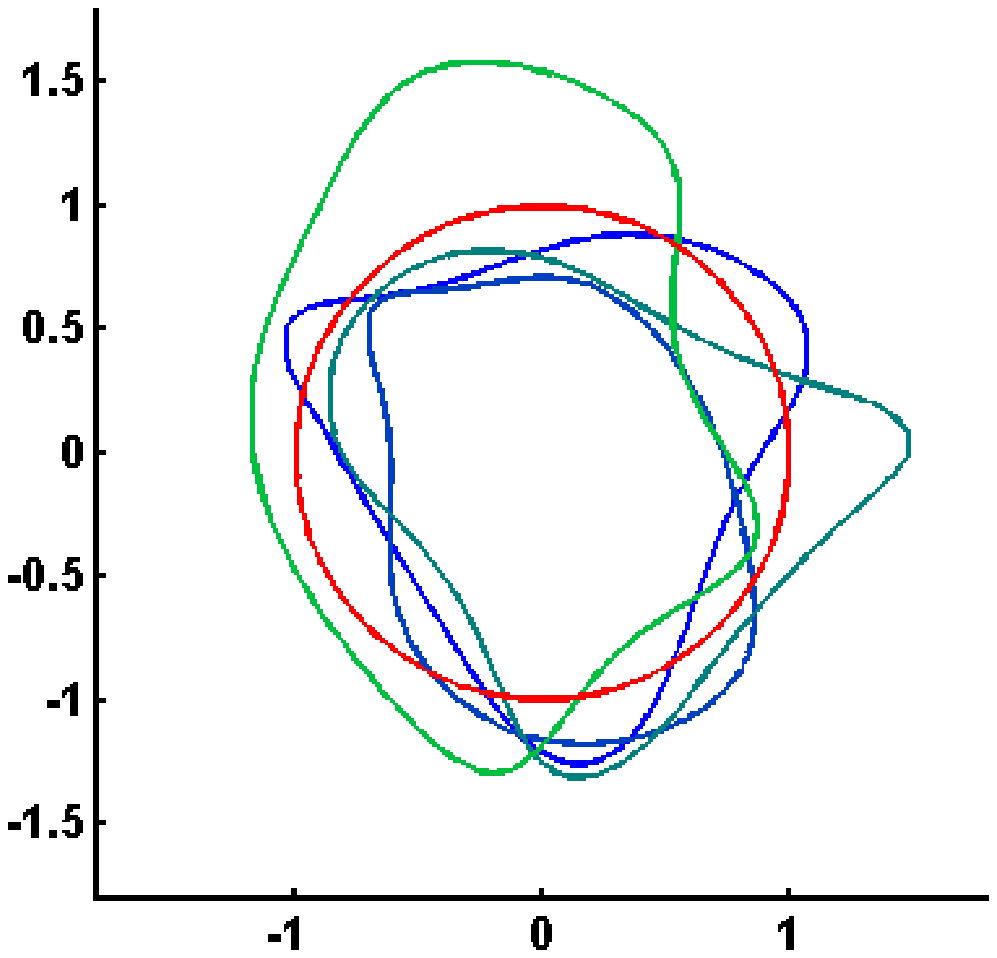}
 \caption{\footnotesize{$5$ simulations of random deformations of the circle.}}
 \label{vois}
\end{minipage} \hfill
\begin{minipage}{.5\linewidth}
\centering\includegraphics[width=7cm]{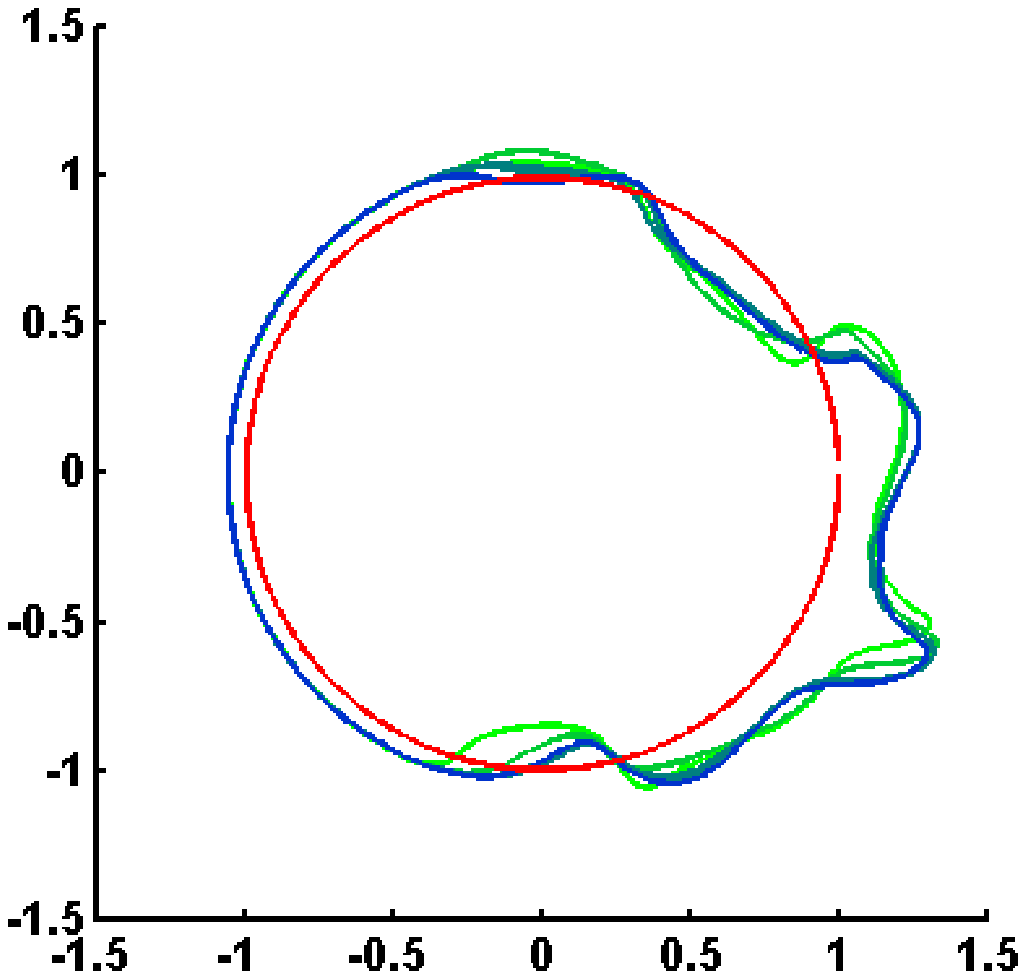}
\caption{\footnotesize{Local noise.}}
  \label{locale}
 \end{minipage} \hfill
\end{figure}

Theorem \ref{theo_first_model} enables us to deal with a white noise with respect to the induced measure of a manifold embedded in $\ms{R}^k$.  
\begin{prop}
There exists $\phi:T_2 \mapsto S_2$ verifying the assumptions of proposition \ref{reparameter}. If $B_t$ is the cylindrical Brownian motion for the measure associated with the induced metric of $S_2$ in $\ms{R}^3$, then $\phi^*(B_t)$ can be written as $\sqrt{\Jac(\phi)}dW_t$ with $W_t$ a cylindrical Brownian motion. 
\end{prop}

\begin{proof}
Consider a dyadic partition of $T_2$ in $6$ squares. The radial projection of the cube on the sphere gives the desired result. The map $\phi$ is given by the mapping of the dyadic partition on the six faces of the cube, which is piecewise smooth and the Jacobian is bounded below. Then we have the desired result thanks to the corollary \ref{corr1}.
\end{proof}

\begin{figure}[htbp]
 \begin{minipage}{.45\linewidth}
  \centering	\includegraphics[width=7cm]{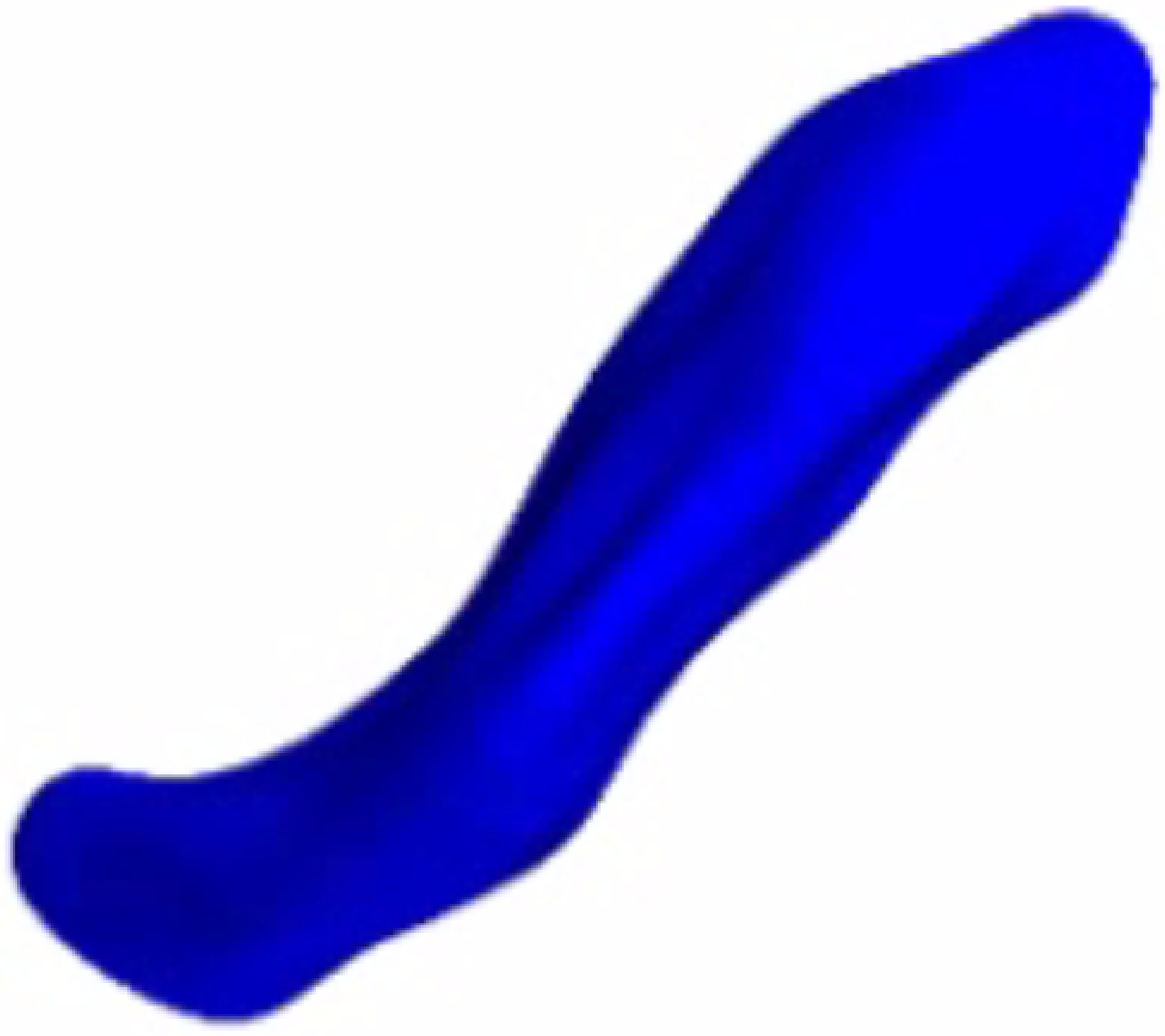}
 \caption{\footnotesize{The initial hippocampus}}
 \label{initialhippo}
\end{minipage} \hfill
\begin{minipage}{.5\linewidth}
\centering\includegraphics[width=7cm]{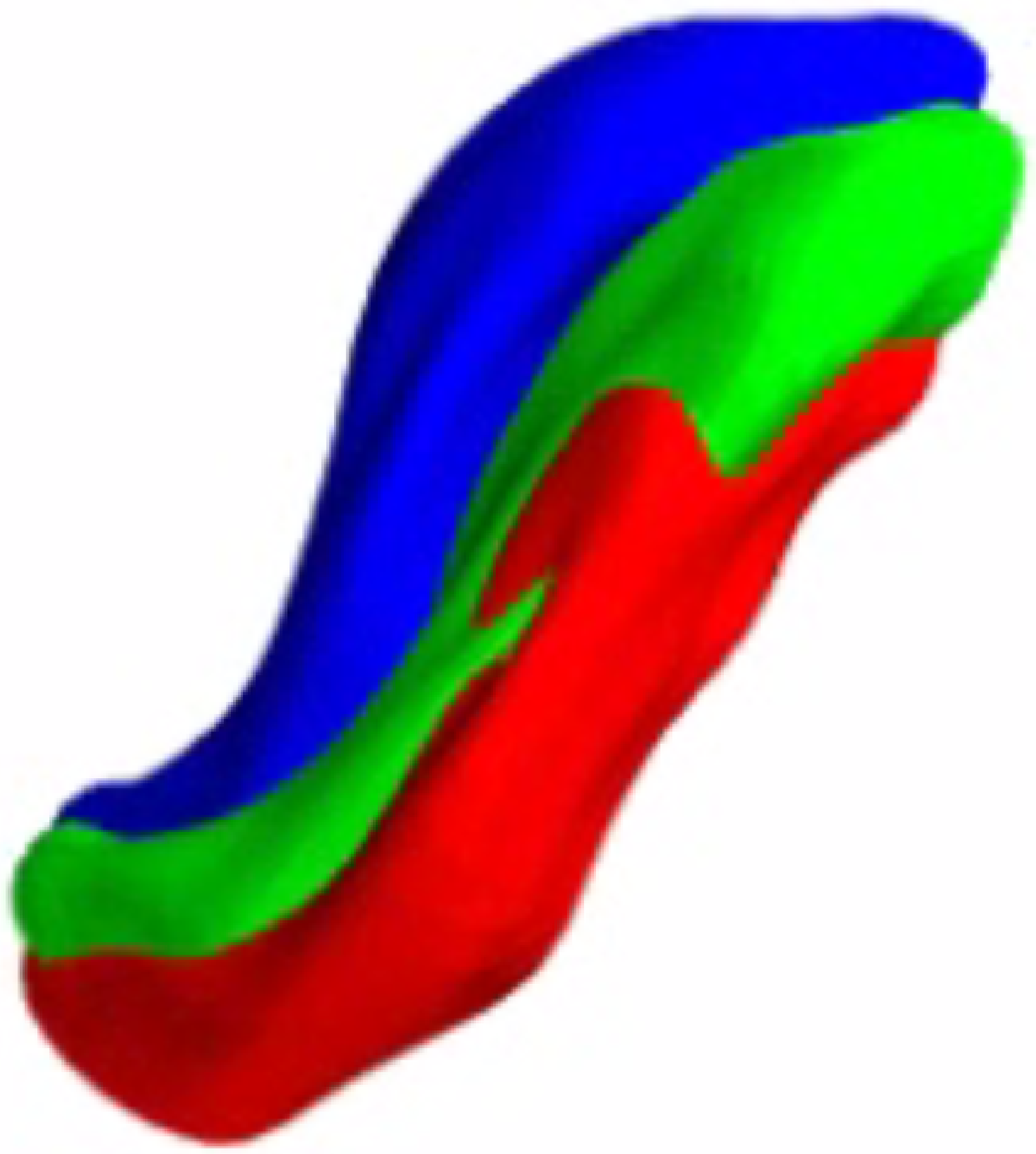}
\caption{\footnotesize{$3$ perturbations of the geodesic shooting}}
  \label{voishippo}
 \end{minipage} \hfill
\end{figure}

As said above, following such kind of decomposition we can get a wide range of embedded manifolds in the euclidean space. To illustrate the model in $3$ dimensions, we give some examples of the stochastic shooting between an initial hippocampus\footnote{data courtesy of G. Gerig (University of Utah) used in \cite{Gerig:data} from a study on autism disease} (the so called part of the brain located in the medial temporal lobe) and a target one. The figure Fig.~\ref{initialhippo} represents the initial hippocampus and the figure Fig.~\ref{voishippo} shows in the same figure $3$ simulations (red, green and blue) of the SDE with an initial momentum that solves the boundary value problem between the initial hippocampus and a target hippocampus not showed here.

In the simulations we observed that a statistical study of the stochastic model requires to control carefully the invariance of the system with respect to a change of time coordinates and the understanding of the relation between the kernel size and the variance term of the model.

\section{Conclusion and open perspectives} \label{conclusion}

The original motivation of this work was to prove an extension of the stochastic model in the case of landmarks to the infinite-dimensional case of shapes. In Section \ref{landmarkcaseSection} we proved that the solutions of the SDE on landmarks are defined for all time by controlling the energy of the system with the help of It\^{o} formula on the Hamiltonian. Hence it gives a stochastic shape evolution model in the case of landmarks.
We then developed in Subsection \ref{prop_needed} a strategy to extend the SDE in infinite dimension to a well chosen Hilbert space. We discussed in Section \ref{F_s} an example of such a well chosen Hilbert space by introducing the spaces $F_s$ in any dimension. As we aimed to prove a convergence of the finite-dimensional case of landmarks to the case of shapes, we also gave a presentation of the cylindrical Brownian motion and the related It\^{o} integral that really suits our needs. Apart from our particular choice for the spaces $F_s$, we proved in Section \ref{generalsolutions} under general hypothesis (Section \ref{Approx_lemmas}) on the finite-dimensional approximation subspaces that the solutions of the SDE on these subspaces converge almost surely to the solutions of the SDE in infinite dimension. Finally, we dealt with a general variance term to account for a possible reparameterization of the shape as detailed in Section \ref{applications}, where some simulations in 2D and 3D are showed.

On the mathematical aspects of this work, we did not explore yet all the possibilities for the structure of noise in our framework. At this point an operator $\sigma \in L(P,P)$ (with the $L^1$ condition) seemed to be sufficient for practical applications. But for instance we would like to know if the situation where the noise is supported by a finite sum of Dirac measures belongs to our framework for a white noise on $L^2(T_n,\mu)$ with $\mu$ the Lebesgue measure. Moreover we guess that our work can be extended to any Radon measure on $T_n$ instead of the Lebesgue measure, which would extend the structures of the noise that can be attained.

Another mathematical perspective opened with this work is to study stochastic perturbations of the EPDiff equation:
\begin{equation}
\begin{cases}
dm + ad_{u}^*m \, dt= \sigma dB_t \, , \\
u = K\star m\,,
\end{cases}
\end{equation}
where $K$ is the chosen kernel. However, we face the problem of the definition of the noise on the space of momentum which is the dual of the Hilbert space of vector fields $V$. As a consequence the choice on the noise is really broad and we underline that in our case the manifold on which the diffeomorphism group acts gives the structure of the noise.

Our central motivation with these stochastic second-order model is to design growth model on shape spaces. Enhancements of this model are at hand at least in two directions: the first one is to incorporate a deterministic control variable (absolutely continuous) on the evolution of the momentum to fit in the splines framework presented in the introduction section. The other direction is to incorporate in the random noise a jump process to account for sudden transformations of the shape. Therefore the enhanced model could be written as:

\begin{equation} \label{generalized_stochastic_system}
\begin{cases}
dp_t = -\partial_q H(p_t,q_t) + u_t + \ve d B_t  + dJ_t\\
dq_t = \partial_p H(p_t,q_t) \, ,
\end{cases}
\end{equation}
with $J_t$ for instance a compound Poisson process $J_t = \sum_{i=1}^{N(t)} j(i)$ with $j(i)$ independent and identically distributed random variables on $P=F_{-s}$ independent of the Poisson process. This jump process gives the opportunity to introduce discontinuities in the momentum evolution.

The parameterization of the noise is a crucial issue, since the main issue in this model is a certain redundancy between the kernel and the choice of the noise. Also this parameterization is strongly related to the observed data. For instance, if evolutions of surfaces are considered, the corresponding momentums are always orthogonal to the current shape. Therefore the noise introduced needs to keep this structure unchanged. More generally, the effect of the noise should keep the symmetries of the deterministic solutions.
\\
Last, the role of the time variable is also crucial in the estimation of the time variability of a biological organ. That's why we should study stochastic models able to retrieve the information on the speed of the evolution. Although the speed is encoded in the stochastic model, the model could be able to learn a time reparameterization. 
It is a first step in the direction of designing a realistic growth model for shapes within the framework of \emph{large deformation diffeomorphisms} and it underlines the efficiency of second-order models as good candidates. 
\\
Being aware of some developments for the estimation of the diffusion parameters for second-order models in \cite{RePEc:bla:jorssb:v:71:y:2009:i:1:p:49-73} or in \cite{PokernPhD}, our main research efforts will be focused on consistent statistical schemes to be applied on biomedical data.

\section*{acknowledgements}
I am grateful to Alain Trouvé for important contributions and Darryl D. Holm for his support.

\section{Appendix}
\begin{prop} \label{tensor_product_algebras}
  Let $(E_i)_{1\leq i\leq 2}$ be two separable RKHS of real valued
  functions defined respectively on $\mathcal{X}_1$ and
  $\mathcal{X}_2$. Then we have~:
  \begin{enumerate}
  \item The tensor space $E_1\otimes E_2$ is a separable RKHS on
    $\mathcal{X}_1\times \mathcal{X}_2$.
  \item If for any $i\in \{1,2\}$, there exists $M_i>0$ such that for
    any $f,f'\in E_i$, $ff'\in E_i$ and $|ff'|_{E_i}\leq
    M_i|f|_{E_i}|f'|_{E_i}$ (we say that $E_i$ is an algebra of
    functions with continuous product), then the same result is true
    for $E_1\otimes E_2$. More precisely, for any $g,g'\in E_1\otimes
    E_2$, we have $gg'\in E_1\otimes E_2$ and
$$|gg'|_{E_1\otimes E_2}\leq M_1M_2|g|_{E_1\otimes E_2} |g'|_{E_1\otimes E_2}$$ 
\end{enumerate}
\end{prop}

\begin{proof}
\indent \textbf{Proof of 1)~:} If $(a_n)_{n\geq 0}$ and $(b_n)_{n\geq 0}$ are Hilbert basis of $E_1$ and $E_2$, then $(a_m\otimes b_n)_{m,n\geq 0}$ is an Hilbert basis for $E_1\otimes E_2$. Now for any $(x_1,x_2)\in\mathcal{X}_1\times \mathcal{X}_2$, there exists a unique  continuous map $\delta_{x_1,x_2}:E_1\otimes E_2\to \mathbb{R}$ defined by the values  $\delta_{x_1,x_2}(a_m\otimes b_n)\doteq a_m(x_1)b_n(x_2)$ on the Hilbert basis. Indeed,  for any finite linear combination $g=\sum_{(m,n)\in I}\lambda_{m,n}a_m\otimes b_n$, one has 
$$|\delta_{x_1,x_2}(g)|^2=\left|\left(\sum_{n\geq 0}\left(\sum_{m\text{ s.t.}(m,n)\in I}\lambda_{m,n}a_m(x_1)\right)b_n\right)(x_2)\right|^2\,.$$ 
Since $E_1$ and $E_2$ are two RKHS, there exist 
$C_1(x_1)$ and $C_2(x_2)$ (depending on $x_1$ and $x_2$) such that $|a(x_1)|\leq C_1(x_1)|a|_{E_1}$ and $b(x_2)\leq C_2(x_2)|b|_{E_2}$ for any $a,b\in E_1\times E_2$. Therefore, 
\begin{multline*}
  |\delta_{x_1,x_2}(g)|^2\leq C_2(x_2)^2\sum_{n\geq
    0}\left(\sum_{m\text{ s.t.}(m,n)\in
      I}\lambda_{m,n}a_m(x_1)\right)^2\\
\leq
  C_2(x_2)^2C_1(x_1)^2\sum_{(m,n)\in I}\lambda_{m,n}^2\leq (C_1(x_1)C_2(x_2))|g|_{E_1\otimes E_2}^2\,.
\end{multline*}
In particular, we have
$$\|\delta_{x_1,x_2}\|\leq \|\delta^1_{x_1}\|\|\delta^2_{x_2}\|$$
where $\delta_{x_i}^i:E_i\to\mathbb{R}$ such that $\delta_{x_i}^i(f)=f(x_i)$ for any $f\in E_i$.\\
\noindent
Note finally that for any $g\in E_1\otimes E_2$, if $\delta_{x_1,x_2}(g)=0$ 
for any $(x_1,x_2)\in E_1\times E_2$, then $g=0$. 
Indeed, if $g=\sum_{m,n}\lambda_{m,n}a_m\otimes b_n$ then 
$g=\sum_{n\geq }A_n\otimes b_n$ where $A_n=\sum_{m\geq 0}\lambda_{m,n}a_m$ 
so that for any $(x_1,x_2)\in\mathcal{X}_1\times\mathcal{X}_2$, we have 
$\sum_{n\geq 0}A_n(x_1)b_n(x_2)=\delta^2_{x_2}(\sum_{n\geq 0}A_n(x_1)b_n)=0$. 
For fixed $x_1\in \mathcal{X}_1$, and $b\doteq \sum_{n\geq 0}A_n(x_1)b_n$, 
we have $b(x_2)$ for any $x_2\in\mathcal{X}_2$ so that $b=0$ and $A_n(x_1)=0$ 
for any $n\geq 0$. Considering now arbitrary $x_1$, we get that 
$\sum_{m\geq 0}\lambda_{m,n}a_m=0$ so that $\lambda_{m,n}=0$ and $g=0$. 
\noindent
Hereafter, we will denote 
(without ambiguity) 
$$g(x_1,x_2)\doteq \delta_{x_1,x_2}(g)\,.$$

\textbf{Proof of 2)~:} If $g_N=\sum_{0\leq m,n\geq N}\lambda_{m,n}a_m\otimes b_n$ and $g'_N=\sum_{0\leq m,n\leq N}\lambda'_{m,n}a_m\otimes b_n$, then 
$$(g_Ng'_N)=\sum_{0\leq m,n,p,q\leq N}\lambda_{m,n}\lambda'_{p,q}(a_m\otimes f_n)(a_p\otimes b_q)\,.$$
Since $a_ma_p\in E_1$ and $b_nb_q\in E_2$ and 
$$(a_m\otimes f_n)(x_1,x_2)(a_p\otimes b_q)(x_1,x_2)=a_m(x_1)a_p(x_1)b_n(x_2)b_q(x_2)=(a_ma_p)(x_1)(b_nb_q)(x_2)$$ 
we get $(a_m\otimes f_n)(a_p\otimes b_q)=(a_ma_p)\otimes(b_nb_q)$ and
$$|(a_ma_p)\otimes(b_nb_q)|_{E_1\otimes E_2}=|a_ma_p|_{E_1}|b_nb_q|_{E_2}\leq M_1M_2\,.$$
Hence $|g_Ng'_N|_{E_1\otimes E_2}\leq M_1M_2\sum_{0\leq m,n,p,q\leq N}|\lambda_{m,n}\lambda'_{p,q}|\stackrel{C.S.}{\leq} M_1M_2|g|_{E_1\otimes E_2}|g|_{E_1\otimes E_2}$ and the sequence $g_Ng'_N$ is bounded in $E_1\otimes E_2$. Since for any weak limit, we have $g_{N_k}g'_{N_k}(x_1,x_2)\to (gg')(x_1,x_2)$ we get that $gg'\in E_1\otimes E_2$ and by lower semi-continuity of the norm for the weak convergence~:
$$|gg'|_{E_1\otimes E_2}\leq M_1M_2|g|_{E_1\otimes E_2}|g'|_{E_1\otimes E_2}\,.$$
\end{proof}

\nocite{Lindenstrauss}
\bibliographystyle{abbrv}
\bibliography{SSM}
\end{document}